\documentclass[11pt,reqno]{amsart}

\usepackage[utf8]{inputenc}
\usepackage[margin=1.25in]{geometry}
\usepackage{hyperref}
\usepackage{appendix}
\usepackage{amsfonts}
\usepackage{amssymb}
\usepackage{stmaryrd} 
\usepackage{amsmath}
\usepackage{amsthm}
\usepackage{mathrsfs}

\theoremstyle{plain}
\newtheorem{theorem}{Theorem}[section]
\newtheorem{corollary}[theorem]{Corollary}
\newtheorem{lemma}[theorem]{Lemma}
\newtheorem{proposition}[theorem]{Proposition}
\theoremstyle{remark}

\newtheorem{Definition}[theorem]{Definition}

\newtheorem{remark}[theorem]{Remark}

\usepackage{biblatex} 
\numberwithin{equation}{section}

\title{Solving McKean-Vlasov SDEs and SPDEs via relative entropy}
\author{Yi Han}
\address{Department of Pure Mathematics and Mathematical Statistics, University of Cambridge.
}
\email{yh482@cam.ac.uk}
\thanks{Supported by EPSRC grant EP/W524141/1.}

\begin{document}

\maketitle

\begin{abstract} In this paper we explore the merit of relative entropy in proving weak well-posedness of McKean-Vlasov SDEs and SPDEs, extending the technique introduced in Lacker \cite{lacker2021hierarchies}. In the SDE setting, we prove weak existence and uniqueness when the interaction is path dependent and only assumed to have linear growth. Meanwhile, we recover and extend the current results when the interaction has Krylov's $L_t^q-L_x^p$ type singularity for $\frac{d}{p}+\frac{2}{q}<1$, where $d$ is the dimension of space. We connect the aforementioned two cases which are traditionally disparate, and form a solution theory that is sufficiently robust to allow perturbations of sublinear growth at the presence of singularity, giving rise to the well-posedness of a new family of McKean-Vlasov SDEs. Our strategy naturally extends to the cases of a fractional Brownian driving noise $B^H$ for all $H\in\left(0,1\right)$, obtaining new results in each separate case $H\in\left(0,\frac{1}{2}\right)$ and $H\in\left(\frac{1}{2},1\right)$. In the SPDE setting, we construct McKean-Vlasov type SPDEs with bounded measurable coefficients from the prototype of  stochastic heat equation in spatial dimension one, and we do the same construction for the stochastic wave equation and a SPDE with white noise acting only on the boundary. In addition, we generalize some quantitative propagation of chaos results for SDEs into the SPDE setting. 
\end{abstract}

\section{Introduction}
The McKean-Vlasov SDE \begin{equation}\label{mckean11}
d X_{t}=\left(b_{0}\left(t,X_{t}\right)+\left\langle\mu_{t}, b\left(t,X_{t}, \cdot\right)\right\rangle\right) d t+d W_{t}, \quad \mu_{t}=\operatorname{Law}\left(X_{t}\right) 
\end{equation}
where $W$ is a Brownian motion, has attracted lots of attention in recent years. It describes the $n\to\infty$ limit of the interacting particle system
\begin{equation}\label{theparticlesystem}d X_{t}^{n, i}=\left(b_{0}\left(t,X_{t}^{n, i}\right)+\frac{1}{n-1} \sum_{j \neq i} b\left(t,X_{t}^{n, i}, X_{t}^{n, j}\right)\right) d t+d W_{t}^{i},\quad i=1,\cdots,n,\end{equation} where $W^1,\cdots,W^n$ are $n$ independent Brownian motions.

This paper is about solving the McKean-Vlasov SDE \eqref{mckean11}. When $b_0$ and $b$ are time independent, Lipschitz and bounded, it is classical that  \eqref{mckean11} has a unique strong solution, see \cite{10.1007/BFb0085169}. When $b$ is nonsmooth or even not locally bounded, as in many concrete physical settings, the problem becomes challenging and there has been a few early works, and 
significant recent effort in solving \eqref{mckean11}, such as \cite{shiga1985central} \cite{scheutzow1987uniqueness}, \cite{jourdain1997diffusions},\cite{bossy2005some}, \cite{mishura2020existence}, \cite{huang2019distribution},
\cite{lacker2018strong} and \cite{rockner2021well}.

Finding a weak solution to this equation boils down to the following fixed point problem: for $0\leq t\leq T$ let $\mathcal{C}_t^d:=\mathcal{C}([0,t];\mathbb{R}^d)$ denote the continuous path space endowed with the supremum norm $\|X\|_t=\sup_{0\leq s\leq t
}|X_s|$, and write $\mathcal{P}(\mathcal{C}_t^d)$ for the space of Borel probability measures on $\mathcal{C}_t^d$. For a given measure $\mu\in\mathcal{P}(\mathcal{C}_T^d)$, for each $t\in[0,T]$ denote by $\mu_t$ the time-$t$ marginal of $\mu$, and set 
\begin{equation}\label{drift}\bar{b}_{\mu}(t, x):=b_{0}(t, x)+\langle\mu_t, b(t, x, \cdot)\rangle,\quad x\in\mathbb{R}^d.\end{equation}
Denote by $\Phi(\mu)$ the law on $\mathcal{C}_T^d$ of the SDE
\begin{equation}\label{fixedmeasure}dX_t=\bar{b}_\mu(t,X_t)dt+dW_t,\quad \operatorname{Law}(X_0)=\mu_0.\end{equation}

Then any $\mu\in\mathcal{P}(\mathcal{C}_T^d)$ satisfying $\Phi(\mu)=\mu$ is a weak solution to the McKean-Vlasov SDE \eqref{mckean11} with initial law $\mu_0$, and vice versa.

To find a fixed point of $\Phi$, we often fix a distance on  $\mathcal{P}(\mathcal{}{C}_T^d)$ and analyze the continuity of $\Phi$ with respect to this distance. In this analysis we need detailed information on the dependence of $\bar{b}_\mu$ on $\mu$. We may succinctly rewrite the drift $\langle \mu,b\left(t,x,\cdot\right)\rangle$ as $B(t,x,\mu)$ for $\mu\in\mathcal{P}(\mathbb{R}^d)$, with a slight abuse of notation.

Motivated by the recent progress in distribution dependent SDEs such as \cite{coghi2020pathwis} and \cite{galeati2021distribution}, various assumptions have been made on the continuity of $B$ with respect to $\mu$. A common assumption is that $B$ is Lipschitz continuous with respect to $\mu$ in the Wasserstein distance $\mathcal{W}_1(\mathcal{P}(\mathbb{R}^d))$, uniformly in $t$ and $x$. If moreover $B$ is Lipschitz continuous in $x$, it is proved in \cite{coghi2020pathwis} that the McKean-Vlasov SDE \eqref{mckean11} has a unique weak solution even if the driving Brownian motion is replaced by any additive noise  $W$ with $\operatorname{Law}(W)\in\mathcal{P}(\mathcal{C}_T^d)$. 

Another common assumption is that $B$ is continuous with respect to $\mu$ in total variation. This is the case for bounded measurable $b$, where one needs to utilize the regularizing properties of Brownian motion as the corresponding ODE $dX_t=a(t,X_t)dt$ is not well-posed in general if $a:[0,T]\times\mathbb{R}^d\to\mathbb{R}^d$ is merely bounded measurable. In the case $b_0$ and $b$ are bounded measurable, the well-posedness of \eqref{mckean11} has been established, see for example \cite{lacker2018strong}, Theorem 2.4 or Proposition 7.1 of \cite{lacker2021hierarchies}. We outline the proof in \cite{lacker2021hierarchies}: for any $\mu\in\mathcal{P}(\mathcal{C}_T^d)$ and any $t\in[0,T]$, denote by $\mu[t]$ the projection of $\mu$ to $\mathcal{C}_t^d$. Then for any $\mu,\nu\in\mathcal{P}(\mathcal{C}_T^d),$
\begin{equation}\label{mainproof}
\begin{aligned}
\|\Phi(\mu)[t]-\Phi(\nu)[t]\|_{\mathrm{TV}}^{2} & \leq 2 H(\Phi(\mu)[t] \mid \Phi(\nu)[t]) \\
&=\int_{\mathcal{C}_{t}^{d}} \int_{0}^{t}\left|\bar{b}_{\mu}(s, x)-\bar{b}_{\nu}(s, x)\right|^{2} d s \Phi(\mu)[t](d x) \\
&=\int_{\mathcal{C}_{t}^{d}} \int_{0}^{t}|\langle\mu_s-\nu_s, b(s, x, \cdot)\rangle|^{2} d s \Phi(\mu)[t](d x) \\
& \leq\left\||b|^{2}\right\|_{\infty} \int_{0}^{t}\|\mu[s]-\nu[s]\|_{\mathrm{TV}}^{2} d s.
\end{aligned}
\end{equation}
Here $\|\cdot\|_{TV}$ denotes total variation distance and $H\left(\cdot\mid\cdot\right)$ denotes relative entropy. In the first line of \eqref{mainproof} we used Pinsker's inequality \eqref{variational} and in the second line we used Girsanov transform \eqref{Browniancase}. A standard Picard iteration finishes the proof. \footnote{
We note that the estimate \eqref{mainproof} holds without change if the drift is of the general form $B(t,x,\mu)$ such that $B$ is Lipschitz continuous in $\mu$ in total variation distance, uniformly in $t$ and $x$. We can switch to this nonlinear setting in many, but not all, cases of the present article, which is made clear in Subsection \ref{subsection1.0.1}.}

The finding of this paper is that the seemingly simple argument \eqref{mainproof} in the bounded measurable case can be generalized to solve \eqref{mckean11} with many different types of interactions that were previously handled by other means or even not known to have a solution. An application of arguments like \eqref{mainproof} not only makes the proof much shorter but also allows more flexibility in the choice of $b_0$ and $b$.

\subsubsection{Interactions of linear growth}\label{subsection1.0.1} A first novelty of our relative entropy approach is that we can consider coefficients that locally, but not globally, belong to standard Besov-Hölder spaces. This is the case for a drift or interaction function that has linear growth, i.e. for some $K>0$ we have $|b(t,x,y)|\leq K(1+|x|+|y|)$. This function clearly does not belong to $L^\infty(\mathbb{R}^d)$ unless we impose a weight function. One can observe that controlling the drift $\langle \mu_t, b(t,x,\cdot)\rangle$ in terms of $\mu_t$ will get much harder when $b$ does not belong to these function spaces, which is a problem arising in the McKean-Vlasov setting but not in the SDE setting $dX_t=b(X_t)dt+dW_t$. It turns out that relative entropy enables us to take averages appropriately, and gives a satisfying answer in the linear growth case.

In most references on McKean-Vlasov equations, we consider drifts that depend on densities in a nonlinear way, written in the notion $B(t,x,\mu)$, assuming that $B(t,x,\cdot)$ is continuous in $\mu$ in Wasserstein distance or total variation, uniformly in $t$ and $x$. This does not seem to be the case for an interaction $\langle \mu_t, b(t,x,\cdot)\rangle$ where $b$ has linear growth, and that is the reason why in all the examples with linear growth coefficients, we assume the interaction is given in a linear form like $\langle \mu_t,b(t,x,\cdot)\rangle$ but not in a nonlinear form like  $B(t,x,\mu)$. More discussions on this subtle discrepancy will be given in Remark \ref{linearmore}.

Under this linear growth condition, the strategy of solving \eqref{mckean11} is to combine \eqref{mainproof} with a truncation argument, making use of weighted Pinsker's inequality that helps us to accomplish an averaging procedure (actually, we will average twice.) Working with relative entropy is crucial in the argument, and is currently the only possible way to solve these equations. We learn this method from Lacker \cite{lacker2021hierarchies} where he proved a result that is close to Theorem \ref{theorem01}.  Our contribution in Theorem \ref{theorem01} is that we introduce one more averaging procedure, thus relaxing a condition in Proposition 7.2 of \cite{lacker2021hierarchies}. Consequently, we can cover interactions that are only assumed to have linear growth. This result is optimal, in that the linear growth assumption is what is generally assumed on the drift for the finite dimensional SDE $dX_t=b(X_t)dt+dW_t$ to have a unique solution.

\subsubsection{Singular drifts} Another feature of the relative entropy technique is we can handle drifts that satisfy an integrability condition.

In these situations, the drift $B$ has continuity with respect to $\mu$ only after it is integrated in space. Such interactions have been well studied since Krylov \cite{krylov2005strong}. As an example, we may have that for some $C_B>0$,
\begin{equation}\label{shisandian}\left\| B(t,\cdot,\mu_1)-B(t,\cdot,\mu_2)\right\|_{L^p(\mathbb{R}^d)}\leq C_B\|\mu_1-\mu_2\|_{TV}.\end{equation}

If $b$ satisfies Krylov's integrability condition (see the condition on $b_1$ in Theorem \ref{lplqfirst}) , the McKean-Vlasov SDE \eqref{mckean11} has been solved in \cite{rockner2021well} via PDE methods and a compactness argument. Yet a simple probabilistic method is available: we can use a similar argument as in \eqref{mainproof} by exploiting  \eqref{shisandian} and some auxiliary estimates in the specific case (for example, Krylov's estimate in Theorem \ref{lplqfirst}). We now form a connection between the linear growth case and the case of a singular interaction satisfying integrability conditions, as we can solve the two separate cases within the same framework \eqref{mainproof}. This enables us to consider interactions  $B=B_1+B_2$, each of a different type (see Theorem \ref{wes}). To our best knowledge, \eqref{mckean11} with such mixed interactions have not been well studied in the literature.

The relative entropy technique \eqref{mainproof} appears to be more flexible as we can have some sublinear perturbation of the drift or the interaction in Theorem \ref{lplqfirst} and \ref{wes}, while the method based on Zvonkin's transform in \cite{rockner2021well} seems to be more stringent.

\subsubsection{Non-Markovian processes}
 Fractional Brownian motion is a non-Markovian process and is not a semimartingale, so most techniques in stochastic analysis cannot be applied to solve the fractional Brownian analogue of \eqref{mckean11}. However, the Girsanov transform of fractional Brownian motion can be applied, and we may still use the framework \eqref{mainproof} in this fractional setting. The results we obtain appear to be sharper than the existing literature, such as \cite{bauer2019mckean} and \cite{galeati2021distribution} as we can consider interactions that have a growth at infinity, whereas these works focused on the bounded case (bounded in some Hölder-Besov spaces on $\mathbb{R}^d$.) The case of distributional drift in the $H\in(0,\frac{1}{2})$ case, recently obtained in \cite{galeati2021distribution}, can also be reproved by similar techniques and we give a fairly short proof of it.

 \subsubsection{Stochastic PDEs}
 Consider the stochastic heat equation on the interval $[0,1]$
 \begin{equation}\label{mildspde}
 \frac{\partial}{\partial t} X(t,\sigma)=\frac{\partial^2}{\partial\sigma^2}  X(t,\sigma)dt+f(t,\sigma, X(t,\sigma))dt+dW(t,\sigma),t\geq 0,\quad \sigma\in[0,1],\end{equation}
 with Neumann boundary condition on $[0,1]$, $W$ a cylindrical Wiener process on $L^2([0,1];\mathbb{R})$, and $f$ is some bounded measurable drift which will be modified later. We consider mild solutions to the SPDE \eqref{mildspde} given as follows (see for example \cite{da2014stochastic}): given an initial data $X(0)\in C_0([0,1];\mathbb{R})$, there is a unique solution $X(t,\sigma)$ satisfying
\begin{equation} \label{mild11}
    X(t)=S(t)X(0)+\int_0^t S(t-s)F(s)ds+\int_0^t S(t-s)dW(s),
\end{equation}
where $F(s)(\sigma):=f(s,\sigma,X(s,\sigma))\in L^2([0,1];\mathbb{R})$ and $S(t)$ is the semigroup generated by $L=\frac{\partial^2}{\partial\sigma^2}$ on $L^2([0,1];\mathbb{R})$ with Neumann boundary condition. The existence and uniqueness of a probabilistic strong solution to \eqref{mild11} with bounded measurable $f$ is established in \cite{gyongy1993quasi}.

In this article we propose to consider a McKean-Vlasov type analogue of \eqref{mildspde} inspired by Hao Shen and Weinan E \cite{shen2013mean}, who worked with Lipschitz coefficients. Let $H$ denote the Hilbert space $L^2([0,1];\mathbb{R})$ and $\mathcal{P}(H)$ denote the space of probability measures on $H$, we consider a measurable function 
$G:[0,T]\times H\times \mathcal{P}(H)\to H$ satisfying, for some $M>0$,
$$\|G(t,x,\mu)-G(t,x,\nu)\|_H\leq M \|\mu-\nu\|_{TV}$$
uniformly over $t\in[0,T]$, $x\in H$ and $\mu,\nu\in\mathcal{P}(H)$, and moreover that $G$ is bounded, $\|G\|_H\leq M$. Then we will be able to solve the following distribution dependent SPDE
\begin{equation}\label{eq1.90}
 \frac{\partial}{\partial t} X(t)=\frac{\partial^2}{\partial\sigma^2}  X(t)dt+G(t,X(t),\mu_t)dt+dW(t),\quad \operatorname{Law}(X_t)=\mu_t.\end{equation}

 By the definition of total variation distance, one can check that \eqref{eq1.90} covers the following SPDE
 \begin{equation}\label{speciclcaseweinan}
     \frac{\partial }{\partial t}X(t,\sigma)=\frac{\partial^2}{\partial\sigma^2}X(t,\sigma)dt+\int_0^1 f(X(t,\sigma),z)\rho_{X(t,\sigma')}(dz)d\sigma' dt  +dW_t,\quad\sigma\in[0,1],
 \end{equation}
 where $f:\mathbb{R}\times\mathbb{R}\to\mathbb{R}$ is some bounded measurable function. In \cite{shen2013mean} the authors assumed $f$ to be Lipschitz, with the physical assumption that $f=-\nabla V$ for some potential $V$ that models an interaction potential between polymers, and that $V\in \mathcal{C}^2$. In our paper we can consider $V$ to be merely $\mathcal{C}^1$, and $f$ bounded measurable.

Our proof technique is not limited to the stochastic heat equation, and we will carry out the same construction for McKean-Vlasov type stochastic wave equation, and a McKean-Vlasov type SPDE where the white noise only acts on the boundary.
 
As another novel application of our construction, the quantitative propagation of chaos result obtained in \cite{lacker2021hierarchies} can be extended to this SPDE setting. That is, given a bounded measurable map $F:[0,T]\times H\times H\to H$ and consider the interacting particle system with $N$ particles given i.i.d. initial laws
$$\frac{\partial}{\partial t}X^i(t)=\frac{\partial^2}{\partial\sigma^2} X^i(t)dt+\frac{1}{N}\sum_{j=1,j\neq i}^N F(t,X^i(t),X^j(t))dt+dW^i(t),\quad i=1,\cdots,N,$$
then the law of each particle $X^i$ converges to the law of the limiting SPDE 
\begin{equation}
    \frac{\partial}{\partial t}X(t)=\frac{\partial^2}{\partial\sigma^2} X(t)dt+\langle \mu_t, F(t,X(t),\cdot)\rangle dt+dW(t),\quad\operatorname{Law}(X_t)=\mu_t,
\end{equation}
with a $O(\frac{1}{n^2})$ convergence rate in relative entropy. This convergence rate is known to be sharp in the finite dimensional SDE setting. The proof is given in Section \ref{chaos7.2}.
 
\subsection{List of main results} 

To simplify presentations, we have deferred  notions regarding path space, progressive measurability, Girsanov transform, Pinsker's inequality and entropy to Appendix \ref{appendixas}. Readers may consult the appendix if some notations appear to be unfamiliar.

In this paper $\mathcal{C}_T^d:=\mathcal{C}([0,T];\mathbb{R}^d)$ denotes the space of $\mathbb{R}^d$-valued continuous paths. We will consider path-dependent interactions, i.e., the interaction $b(t,x,y)$ is defined on $[0,T]\times\mathcal{C}_T^d\times\mathcal{C}_T^d$, and the drift $b_0(t,x)$ defined on $[0,T]\times\mathcal{C}_T^d$. In this path dependent case, we further require that $b_0$ and $b$ are progressively measurable, so the function $b(t,x,y)$ only depends on trajectories of $x$ and $y$ up to time $t$. We also consider the cases where $b$ and $b_0$ are only state dependent, i.e., $b$ is defined on $[0,T]\times\mathbb{R}^d\times\mathbb{R}^d$ and $b_0$ on $[0,T]\times\mathbb{R}^d$. For any probability measure $\mu\in\mathcal{P}(\mathcal{C}_T^d)$, we denote by $\mu_t\in\mathcal{P}(\mathbb{R}^d)$ its marginal law at time $t$. In the path dependent case, progressive measurability ensures that $\langle \mu, b(t,x,\cdot)\rangle$ is well-defined for any $x\in\mathcal{C}_T^d$ and $\mu\in\mathcal{P}(\mathcal{C}_T^d)$. In the state dependent case we will only consider interactions of the form $\langle \mu_t, b(t,x,\cdot)\rangle.$  As these notations will be used interchangeably, we introduce a table that clarifies the precise notation that will be used in each theorem (but see Remark \ref{fywanghx} for an extension.) 

\begin{center}
\begin{tabular}{|c|c|}%
\hline  
Notations & Theorems
\\
\hline  
Path dependent interaction $\langle \mu,b(t,x)\rangle$ & Theorem \ref{theorem01}, \ref{theorem05}.\\
\hline 
State dependent interaction $\langle \mu_t,b(t,x)\rangle$&  Theorem \ref{lplqfirst},\ref{wes},\ref{thm3} 

.\\
\hline
Nonlinear density dependence
& Proposition \ref{proposition1.50}, Theorem \ref{theorem1.71.7}.\\
\hline 
\end{tabular}
\end{center}
 
 We introduce another notion for measures $\mu\in\mathcal{P}(\mathcal{C}_T^d)$. For $\mu\in\mathcal{P}(\mathcal{C}_T^d)$ and any $t\in[0,T]$, denote by $\mu[t]$ the projection of $\mu$ onto $\mathcal{C}_t^d$, that is, we discard everything after time $t$. In some cases $\mu_t$ and $\mu[t]$ will be used interchangeably, and perhaps the best way to remember their difference is by noting that $\mu_t\in\mathcal{P}(\mathbb{R}^d)$ and $\mu[t]\in\mathcal{P}(\mathcal{C}_t^d)$.

\subsubsection{Results for Brownian motion}
The first result we obtain is as follows: 
\begin{theorem}\label{theorem01} Assume that $(b_0,b)$ are path dependent, progressively measurable and 
\begin{itemize}
\item  $\left(b_{0}, b\right)$ satisfy, for some $0<K<\infty$,
\begin{equation}\label{linear}\left|b_{0}(t, x)\right|+|b(t, x, y)| \leq K\left(1+\|x\|_{t}+\|y\|_{t}\right)\quad \forall t \in[0, T], x, y\in \mathcal{C}_{T}^{d}.\end{equation}
 \item The initial distribution $\mu_{0}\in\mathcal{P}(\mathbb{R}^d)$  satisfies $$\int_{\mathbb{R}^{d}} e^{c_{0}|x|^{2}} \mu_{0}(d x)<\infty\quad \text{ for some } c_0>0.$$ 
 \end{itemize}
 Then the McKean-Vlasov SDE\footnote{
In this example we only consider first order interactions of the form $\langle b(t,x,\cdot),\mu_t\rangle$, but not the nonlinear ones. A very detailed explanation of this confinement is given in Remark \ref{linearmore}.}
 
 \begin{equation}
d X_{t}=\left(b_{0}(t, X)+\langle\mu, b(t, X, \cdot)\rangle\right) d t+dW_t, \quad \mu=\operatorname{Law}(X)
\end{equation}
admits a unique weak solution $\mu$ from $\mu_{0}$ satisfying $\mathbb{E}\left[\left\|X\right\|_T\right]<\infty$.  (Proof given in Chapter \ref{chapter2}.)
\end{theorem}

Theorem \ref{theorem01} generalizes Theorem 2.10 of Lacker \cite{lacker2021hierarchies}, where he assumed an extra condition on $b$: for some $K>0$, $$\label{increment}\left|b(t, x, y)-b\left(t, x, y^{\prime}\right)\right| \leq K\left(1+\|y\|_{t}+\left\|y^{\prime}\right\|_{t}\right)\quad\text{ for all } t\in[0,T],x,y,y'\in\mathcal{C}_t^d.$$This condition is removed in our proof, yet our argument follows closely Lacker \cite{lacker2021hierarchies}. In Remark \ref{linearmore} we give an informal discussion to illustrate how we manage to remove that condition.

The linear growth condition \eqref{linear} on $b_0$ and $b$ in Theorem \ref{theorem01} matches the well-known linear growth condition on the drift for the well-posedness of the SDE (without density dependence), see the discussion at the beginning of Section \ref{chapter2}.

We can also treat interactions that have a local singularity.

Define $$
\mathscr{I}_{1}:=\left\{(p, q) \in(1, \infty): \frac{d}{p}+\frac{2}{q}<1\right\}.$$ For a real valued function $f:[0,T]\times\mathbb{R}^d\to\mathbb{R}^d$, abbreviate $f_t(x)$ as $f(t,x)$, we define $$\|f\|_{L_t^{q}\left([0,T],L_x^{p}(\mathbb{R}^d)\right)}:=\left(\int_0^T \|f_t\|^q_{L^p(\mathbb{R}^d)}dt\right)^{\frac{1}{q}},$$ and $L_t^{q}\left([0,T],L_x^{p}(\mathbb{R}^d)\right)$ consists of all such $f$ with $\|f\|_{L_t^{q}\left([0,T],L_x^{p}(\mathbb{R}^d)\right)}<\infty.$

\begin{theorem}
\label{lplqfirst}
Fix $b_0:[0,T]\times\mathbb{R}^d\to\mathbb{R}^d$, $b_i:[0,T]\times\mathbb{R}^d\times\mathbb{R}^d\to\mathbb{R}^d$ for $i=1,2$, and the initial law $\mu_0$. Assume that the following conditions are satisfied:
\begin{itemize}
    \item The drift $b_0$ satisfies $$|b_0(t,x)|\leq K(1+|x|^\beta) \text{ for some }K>0  ,\beta\in[0,1).$$
    \item The interaction $b_1(t,x,y)$ satisfies, for some $(p_1,q_1)\in\mathscr{I}_1$\footnote{More generally, we may consider nonlinear interactions $B(t,x,\mu)$ satisfying \begin{equation}
    \|B(t,x,\mu)-B(t,x,\nu)\|_{L_x^{p_1}(\mathbb{R}^d)}\leq \ell(t)\|\mu-\nu\|_{TV},\quad \text{for some } \ell\in L^{q_1}([0,T]),\end{equation} uniformly in $t\in[0,T]$ and $x\in\mathbb{R}^d$. Jensen's inequality implies that \eqref{asdfghijk} is a special case of this general criterion. The proof of Theorem \ref{lplqfirst} holds with no change in this nonlinear case.}

    \begin{equation}\label{asdfghijk}|b_1(t,x,y)|\leq h_t(x-y) \text{ for some }h\in L_t^{q_1}\left([0,T],L_x^{p_1}(\mathbb{R}^d)\right).\end{equation}

    \item The interaction $b_2(t,x,y)$ is bounded measurable.
    \item  The initial distribution $\mu_0\in\mathcal{P}(\mathbb{R}^d)$ satisfies
    $$\int_{\mathbb{R}^d} e^{\kappa |x|^{2\beta}}\mu_0(dx)<\infty\quad \text{ for all } \kappa\in\mathbb{R}.$$
\end{itemize}
Then the McKean-Vlasov SDE
 \begin{equation}
d X_{t}=\left(b_{0}(t, X_t)+\langle\mu_t, b_1+b_2(t, X_t, \cdot)\rangle\right) d t+dW_t, \quad \mu_t=\operatorname{Law}(X_t)
\end{equation}
 admits a unique weak solution $\mu$ with initial distribution $\mu_{0}$ that satisfies
 $$\int_0^T \left|b_0(t,X_t)+h_t(X_t)\right|^2dt<\infty\text{ with probability 1 }.$$
 (Proof given in Chapter \ref{chapter3}.)
\end{theorem}
The $b_0=0$ case of Theorem \ref{lplqfirst}
was first proved in Röckner and Zhang \cite{rockner2021well}. We recover their results in the case of additive noise (See Remark \ref{multiplicativenoise} for extension to multiplicative noise). Moreover we can have an additional drift of sublinear growth, which is not covered by techniques based on Zvonkin’s transform in \cite{rockner2021well}. Our proof is purely probabilistic and does not involve truncation arguments. See Remark \ref{twolplp} for another generalization.

Then we introduce a set of new assumptions where we interpolate between local $L_t^q-L_x^p$ type singularity and global sublinear growth. In this manner we can produce a class of McKean-Vlasov SDEs whose existence was not previously known. 
\begin{theorem}\label{wes}
Fix $b_i:[0,T]\times\mathbb{R}^d\times\mathbb{R}^d\to\mathbb{R}^d$ for $i=1,2$ and the initial law $\mu_0$. Assume they satisfy the following conditions:
\begin{itemize}
    \item The interaction $b_1(t,x,y)$ satisfies, for some $(p_1,q_1)\in\mathscr{I}_1$,
    $$|b_1(t,x,y)|\leq h_t(x-y) \text{ for some } h\in L_t^{q_1}\left([0,T],L_x^{p_1}(\mathbb{R}^d)\right).$$ 
    \item The interaction $b_2(t,x,y)$ satisfies \begin{equation}\sup_{t,y}|b_2(t,x,y)|\leq K(1+|x|^\beta)\text{ for some }K>0,    
     \beta\in[0,1).\label{kcondition}\end{equation}
    \item The initial law $\mu_0\in\mathcal{P}(\mathbb{R}^d)$ satisfies  $$\int_{\mathbb{R}^d} e^{\kappa |x|^{2\beta}}\mu_0(dx)<\infty\quad \text{ for all } \kappa\in\mathbb{R}.$$
\end{itemize}
Then the McKean-Vlasov SDE
 \begin{equation}
d X_{t}=\langle\mu_t, b_1+b_2(t, X_t, \cdot)\rangle d t+dW_t, \quad \mu_t=\operatorname{Law}(X_t)
\end{equation}
 admits a unique weak solution $\mu$ with initial distribution $\mu_{0}$ 
 that satisfies  $$\int_0^T\left( K^2|X_t|^{2\beta}+\left|h_t(X_t)\right|^2\right)dt<\infty\text{ with probability 1},$$ where the constant $K$ is given in \eqref{kcondition}.
  (Proof given in Chapter \ref{chapter3}.)
\end{theorem}
In this theorem we have omitted a drift $b_0$ but we can take $b_0$ to be of sublinear growth as we did in Theorem \ref{lplqfirst}, without loss of generality. See Remark \ref{remark3} for a discussion on the condition we set on $b_2$.

\subsubsection{Results for fractional Brownian motion}
Another benefit of our framework is that it extends naturally to the case of SDEs driven by fractional Brownian motion $B^H$, where PDE techniques based on Fokker-Planck equations are hard to apply. We mention \cite{bauer2019mckean} and \cite{galeati2021distribution} as the first works in this direction, and refer to Section \ref{chapter4} for a review of fractional Brownian motion.

In the singular case $H\in(0,\frac{1}{2})$, results of Theorem \ref{theorem01} carry over.
\begin{theorem}\label{theorem05}  Assume that $H\in(0,\frac{1}{2})$ and that the following conditions are satisfied:
\begin{itemize}
\item  $\left(b_{0}, b\right)$ are path dependent, progressively measurable and for some $0<K<\infty$,
\begin{equation}\left|b_{0}(t, x)\right|+|b(t, x, y)| \leq K\left(1+\|x\|_{t}+\|y\|_{t}\right)\quad \forall t \in[0, T], x, y \in \mathcal{C}_{T}^{d}.\end{equation}

 \item The initial distribution $\mu_0\in\mathcal{P}(\mathbb{R}^d)$ satisfies $$\int_{\mathbb{R}^{d}} e^{c_{0}|x|^{2}} \mu_{0}(d x)<\infty\text{ for some } c_0>0 .$$ 
 \end{itemize}
Then the McKean-Vlasov SDE
 
 \begin{equation}
d X_{t}=\left(b_{0}(t, X)+\langle\mu, b(t, X, \cdot)\rangle\right) d t+d B^H_{t}, \quad \mu=\operatorname{Law}(X)\label{mckean1}
\end{equation}
admits a unique weak solution $\mu$ from $\mu_{0}$ satisfying $\mathbb{E}\left[\left\|X\right\|_T\right]<\infty$.  (Proof given in Chapter \ref{chapter5}.)
\end{theorem}

We can also consider the case of a distribution valued drift in the regime $H\in(0,\frac{1}{2})$. The following result has been proved in \cite{galeati2021distribution}, still we provide an alternative, short proof to illustrate the strength of our relative entropy approach. In the following, $B_{\infty,\infty}^\alpha$ denotes the isotropic Hölder-Besov spaces, whose definitions can be found in various places, see for example \cite{galeati2021distribution} and the references therein. As we will use Besov spaces in this example only, we do not recall relevant definitions for sake of conciseness.

\begin{proposition}\label{proposition1.50}
Given a function $B:[0,T]\times \mathbb{R}^d\times\mathcal{P}(\mathbb{R}^d)\to\mathbb{R}^d$ satisfying the following condition: there exists $h\in L^q([0,T],\mathbb{R}_+)$, $q>2$ such that $$\|B(t,\cdot,\mu)\|_{B_{\infty,\infty}^\alpha}\leq h(t)\quad\text{ for any }\mu\in\mathcal{P}(\mathbb{R}^d),$$
and that 
\begin{equation}\label{besovnorms}
\|B(t,\cdot,\mu)-B(t,\cdot,\nu)\|_{B_{\infty,\infty}^\alpha}\leq h(t)\|\mu-\nu\|_{TV},\quad \mu,\nu\in\mathcal{P}(\mathbb{R}^d),\end{equation}
 if moreover the constants $\alpha$, $q$ and $H$ satisfy the following condition
 $$\alpha>1+\frac{1}{Hq}-\frac{1}{2H},$$
then for any $\mu_0\in\mathcal{P}(\mathbb{R}^d)$, there is a unique (probabilistic)\footnote{ In Proposition \ref{proposition1.50} and Theorem \ref{thm3}, well-posedness in the probabilistic strong sense follows from the corresponding results on SDEs once we fix the measure component $\mu$, see for example \cite{catellier2016averaging} and \cite{galeati2021noiseless}. As in this paper we deal with SDEs and SPDEs simultaneously, we stress that the weak and strong solutions in the finite dimensional SDE setting should not be confused with the corresponding notion in the PDE setting.
}strong solution to the McKean-Vlasov SDE
$$dX_t= B(t,X_t,\mu_t) dt+dB_t^H,\quad \operatorname{Law}(X_t)=\mu_t$$ with initial law $\mu_0$.  (Proof given in Chapter \ref{chapter5}.)
\end{proposition}

The assumption \ref{besovnorms} is easy to justify, as for example in the linear case $$B(t,x,\mu)=\int _{\mathbb{R}^D}b(t,x-y)\mu(y)dy$$
for some $b(t,\cdot)\in B_{\infty,\infty}^\alpha$ with $\|b(t,\cdot)\|_{B_{\infty,\infty}^\alpha}\leq h(t).$

In the regular case $H\in(\frac{1}{2},1)$, we prove the following theorem, generalizing Galeati et al. \cite{galeati2021distribution} in that we consider $b_0$ and $b$ to be possibly unbounded, and we remove a condition stated in terms of Besov space norms in \cite{galeati2021distribution}.
\begin{theorem}\label{thm3} Assume that $H\in(\frac{1}{2},1)$, that $(b_0,b)$ are state dependent and that the following conditions are satisfied:
\begin{itemize}
\item With constants $\alpha\in\left(1-\frac{1}{2H},1\right)$ and $\beta>H-\frac{1}{2}>0$,
$$
|b_0(t, x)-b_0(s, y)| \leqslant C\left(|x-y|^{\alpha}+|t-s|^{\beta}\right) \quad \text { for all } s, t \in[0, T], x, y \in \mathbb{R}^{d}.
$$
$$\begin{aligned}&
|b(t, x,x')-b(s, y,y')| \leqslant C\left(|x-y|^{\alpha}+|x'-y'|^{\alpha}+|t-s|^{\beta}\right)\\& \quad  \text { for all }  s, t \in[0, T], x, y,x',y' \in \mathbb{R}^{d}.\end{aligned}
$$

    \item The initial distribution $\mu_0\in\mathcal{P}(\mathbb{R}^d)$ satisfies $$\int_{\mathbb{R}^{d}} e^{c_{0}|x|^{2}} \mu_{0}(d x)<\infty\text{ for some } c_0>0 .$$  \end{itemize}
Then the Mckean-Vlasov SDE
$$dX_{t}=\left(b_{0}(t, X_t)+\langle\mu_t, b(t, X_t, \cdot)\rangle\right) d t+d B_{t}^H, \quad \mu_t=\operatorname{Law}(X_t)
$$
admits a unique (probabilistic) strong solution, whose law $\mu$ with initial law $\mu_0$ satisfying $\mathbb{E}\left[\|X\|_T\right]<\infty.$ (Proof given in Chapter \ref{chapter6}.)
\end{theorem}

The Hölder continuity assumptions on $b_0$ and $b$ in Theorem \ref{thm3} match the assumption on the drift for the corresponding SDE (without distributional dependence) to have a unique strong solution, see Nualart \cite{NUALART2002103}.

We note that some recent results on McKean-Vlasov SDEs driven by fractional Brownian motion are obtained in \cite{galeati2022solution} via stochastic sewing lemma.

\subsubsection{Application to mean field stochastic PDEs} Before stating our main results in the SPDE setting we give a brief literature review and motivating discussions.

Most of the current literature on McKean-Vlasov type stochastic processes are related to finite dimensional SDEs, and not much is done in the infinite dimensional setting, in particular in the setting of stochastic PDEs. To the author's best knowledge, related works in the SPDE setting are \cite{shen2013mean}  and \cite{shen2022large}, see also \cite{hong2022strong} and \cite{kallianpur1995stochastic}, Chapter 9.

We illustrate that our relative entropy technique works equally well in the SPDE setting as they do for finite dimensional diffusion. As a motivating example, we introduce a model in Hao Shen and Weinan E \cite{shen2013mean} that describes interacting polymers. Consider a system of $N$ polymers, each polymer $X_i(t)$ is modelled as a map from $[0,1]$ to $\mathbb{R}^3$, evolving under the dynamics
\begin{equation}\label{particlesystempolymer}
dX_i(t,\sigma)=\tau dW_i(t)+\kappa\frac{\partial^2}{\partial \sigma^2}X_i(t,\sigma)dt+\frac{1}{N}\sum_{j=1}^N \int_0^1 f(X_i(t,\sigma),X_j(t,\sigma'))d\sigma'dt,\end{equation}
where $\sigma\in[0,1]$, $i=1,\cdots,N$, $dW_1,\cdots,dW_n$ are independent space-time white noise, that is, $W_1(t),\cdots,W_n(t)$ are independent cylindrical Wiener processes on $L^2([0,1],\mathbb{R}^3)$. The $\tau$ and $\kappa$ are positive constants that model temperature and elasticity respectively. The function $f:\mathbb{R}^3\times\mathbb{R}^3\to\mathbb{R}^3$ models an interaction force between different polymers. In \cite{shen2013mean} the authors made the assumption that $f$ is Lipschitz continuous and $f(x,x)=0$. Under this assumption they proved that the law of the interacting polymers \eqref{particlesystempolymer} (with i.i.d. initial condition) converges to the law of the self-consistent mean-field equation
\begin{equation}\label{spdemeanfield}
    dY(t,\sigma)=\tau dW_i(t)+\kappa \frac{\partial^2}{\partial\sigma^2}Y(t,\sigma)dt+\int \mathcal{L}_{Y(t)}(dZ)\int_0^1 f(Y(t,\sigma),Z(\sigma'))d\sigma'dt.
\end{equation}

In this paper we show that the Lipschitz continuity assumption on $f$ can be relaxed, and well-posedness of \eqref{spdemeanfield} can be derived under weaker assumptions. 

To begin with, we make the following observation: on the Hilbert space $$H:=L^2([0,1];\mathbb{R}),$$ assume that we are given a measurable function  (where $\mathcal{P}(H)$ denotes the space of Borel probability measures on $H$)
$$G:[0,T]\times H\times\mathcal{P}(H)\to H$$ that is uniformly bounded and Lipschitz continuous in the measure component with respect to total variation distance, i.e., there exists some $M>0$ such that 
\begin{equation}\label{heatnonlinear}
\begin{cases}
    \|G(t,x,\mu)\|\leq M\\ \|G(t,x,\mu)-G(t,x,\nu)\|\leq M \|\mu-\nu\|_{TV},\quad t\in[0,T],\mu,\nu\in\mathcal{P}(H),
\end{cases}
\end{equation}
then the following self-consistent equation on $H$, with $W$ a cylindrical wiener process on $H$,
\begin{equation}\label{selfconsistent}
    \frac{\partial}{\partial t}Y(t)=\frac{\partial^2}{\partial\sigma^2}Y(t)dt+G(t,Y(t),\mu_t)dt+dW(t),\quad\operatorname{Law}(Y(t))=\mu_t,
\end{equation}
 is sufficiently general to cover the case of \eqref{spdemeanfield} with bounded measurable $f$. Though we have made several simplifications in \eqref{particlesystempolymer} by setting $\tau=\kappa=1$ and considering functions valued in one dimension rather than three,  the general case is actually the same.

Our results on McKean-Vlasov SPDEs are summarized in the following theorem. In (1) of the following  we generalize the model of interacting polymers in Hao Shen and Weinan E \cite{shen2013mean} to cover bounded measurable interactions, as discussed in \eqref{speciclcaseweinan} and \eqref{selfconsistent}. The other two cases have not been worked through before.

\begin{theorem}\label{theorem1.71.7}(Proof given in Chapter \ref{chapter 7}.) The following McKean-Vlasov type SPDEs have mild solutions that are unique in the sense of probabilistic weak solutions\footnote{In many cases we have probabilistic strong solutions. For the stochastic heat equation see \cite{gyongy1993quasi}.} 
\begin{enumerate}
    \item (Stochastic heat equation) Given $H:=L^2([0,1];\mathbb{R})$ as above, the drift $G$ satisfying \eqref{heatnonlinear}, and any initial distribution $\mu_0\in\mathcal{P}(H)$, we solve
    \begin{equation}
    \frac{\partial}{\partial t}Y(t)=\frac{\partial^2}{\partial\sigma^2}Y(t)dt+G(t,Y(t),\mu_t)dt+dW(t),\quad\operatorname{Law}(Y(t))=\mu_t,
\end{equation}
      \item (Heat equation with white noise acting on the boundary) In this case set $H:=L^2([0,1];\mathbb{R}^2)$ and $\omega$ a Brownian motion on $\mathbb{R}^2$.
Given a drift $G:H\times \mathcal{P}(H)\to \mathbb{R}^2$ which is uniformly bounded and Lipschitz in the second argument in total variation, that is for some $M>0$, 
\begin{equation}
\begin{cases}
\|G(t,x,\mu)\|_{\mathbb{R}^2}\leq M,\quad x\in H,t\in[0,T],\\
\|G(t,x,\mu)-G(t,x,\nu)\|_{\mathbb{R}^2}\leq M \|\mu-\nu\|_{TV}, \quad  \mu,\nu\in\mathcal{P}(H), x\in H.\end{cases}
\end{equation}
We solve the following self-consistent equation
 \begin{equation}
\begin{cases}
    \frac{\partial u}{\partial t} =\frac{\partial^2 u}{\partial x^2},\quad t>0,x\in [0,1],\\
    (\frac{\partial u}{\partial x}(t,0)=\frac{\partial u}{\partial x}(t,1))=G(t,u(t),\mu_t)+\dot{\omega}(t),\quad t\geq 0,\quad\operatorname{Law}(u(t))=\mu_t,
    \end{cases}
    \end{equation} with initial distribution $\operatorname{Law}(u(0,\cdot))\in\mathcal{P}(H).$
        \item (Stochastic wave equation) Given $H=(L^2(0,1);\mathbb{R})$ as in (1), consider also the space $H^{-1}(0,1)$ the Sobolev space of order -1. Denote by $\widetilde{H}:=L^2((0,1))\oplus H^{-1}((0,1))$ and consider a drift $G:[0,T]\times\widetilde{H}\times\mathcal{P}(\widetilde{H})\to H$ satisfying, for some $M>0$,
    \begin{equation}
        \begin{cases}
            \|G(t,x,y,\mu)\|_H\leq M,\quad (x,y)\in\widetilde{H},\mu\in\mathcal{P}(\widetilde{H}),\\
            \|G(t,x,y,\mu)-G(t,x,y,\nu)\|_H\leq \|\mu-\nu\|_{TV},\quad (x,y)\in\widetilde{H},\mu\in\mathcal{P}(\widetilde{H}).
        \end{cases}
    \end{equation}
    We solve the following self-consistent evolution equation, 
    \begin{equation}
        \frac{\partial^2}{\partial t ^2}Y(t)=\frac{\partial^2}{\partial\sigma^2}Y(t)dt+G(t,Y(t),\frac{\partial}{\partial t}Y(t),\mu_t)dt+dW(t),\quad \operatorname{Law}(Y(t),\frac{\partial}{\partial t}Y(t))=\mu_t
    \end{equation}
for any initial distribution $\operatorname{Law}(Y(0),\frac{\partial}{\partial t}Y(0))\in\mathcal{P}(\widetilde{H}),$ where $W$ is a cylindrical Wiener process on $H$. 
\end{enumerate} 
\end{theorem}

We will discuss propagation of chaos results for SPDEs in Section \ref{chaos7.2}.

\subsection{A list of remarks}
To better justify the scope of this paper, a few important remarks are listed here.

\begin{remark}[Multiplicative noise] \label{multiplicativenoise}The strategy of this paper allows us to consider McKean-Vlasov SDEs with multiplicative noise, that is, with diffusion coefficients $\sigma(t,x)dW$, as long as eigenvalues of $\sigma\sigma^T$ are positive and bounded away from above, uniform over $t$ and $x$.
In computing relative entropy, the terms involved would be $|\sigma^{-1}(b_1-b_2)|^2$ instead of $|b^1-b^2|^2$. 

We have not included new examples with multiplicative noise, partly because in the setting of singular integrable drifts, the presence of multiplicative noise has been well settled in \cite{rockner2021well}, and in the setting of fractional Brownian motion with singular drifts, the case of a multiplicative driving noise is very challenging to deal with (which uses rough paths theory, and we are restricted to $H\in(\frac{1}{3},1)$, see the very recent paper \cite{dareiotis2022path}). Moreover, we do not have weak uniqueness for the SDE itself when the drift is distributional in that setting.  
\end{remark}

\begin{remark}[Second order systems and degenerate noise] The relative entropy approach can also be used when the noise is degenerate, but non-degenerate in the direction where Girsanov transform will be applied. A typical case concerns Kinetic models
$$dX_t=V_t dt,\quad dV_t=b(t,X_t,V_t)dt+dW_t,$$
see for example  Remark 2.11 of \cite{lacker2021hierarchies}. 

In this paper, similar examples in the SPDE setting (stochastic wave equation and heat equation with a white noise acting on the boundary) are provided in Theorem \ref{theorem1.71.7}. 

Finally, we remark that solution theory for second order systems with \textit{singular} coefficients (not belonging to $L^\infty$ at least) is in general much more involved than its first order counterpart, see \cite{zhang2021second}.
\end{remark}

\begin{remark}[Density dependent diffusion coefficient] Our strategy cannot cover diffusion coefficients $\sigma(t,x,\mu)$ that depend on $\mu$ since we make an essential use of Girsanov transform. It is in general a much harder problem to solve nonlinear versions of McKean-Vlasov SDEs like
$$dX_t=b(X_t,\mu_t)dt+\sigma(X_t,\mu_t)dW_t\quad \mu_t=\operatorname{Law}(X_t)$$
with irregular coefficients. Some Hölder continuity assumptions on $b$ and $\sigma$ will be needed to guarantee well-posedness, see for example \cite{de2020strong}. For a link between rough path theory and McKean-Vlasov dynamics, which is prominent for state and distribution dependent diffusion coefficient but is unnecessary in the additive noise case, we mention the recent works \cite{bailleul2020solving} and \cite{bailleul2021propagation}. If one is interested in finding a probabilistic weak solution without justifying its uniqueness in law, one may use compactness arguments as in  \cite{rockner2021well} and \cite{mishura2020existence}, or use the superposition principle as in \cite{barbu2020nonlinear}. See also \cite{hammersley2021mckean} for a method based on Lyapunov function.

\end{remark}

\begin{remark}[More on linear growth coefficients]\label{linearmore} it turns out that there are three levels of complexity when we deal with an interaction $\langle \mu_t, b(t,x,\cdot)\rangle$ in the situation that the function $b$ is of linear growth. In the first level, $ b$ is bounded measurable, then we can consider nonlinear interactions $B(t,x,\mu)$ that is Lipschitz in $\mu$ in total variation distance, of which $\langle\mu_t, b(t,x,\cdot)\rangle$ is just a special case. In the second level, we assume that for some $K>0$, $|b(t,x,y)-b(t,x,y')|\leq K(1+|y|+|y'|)$. In this case it is shown in Lacker\cite{lacker2021hierarchies} via the use of weighted Pinsker's inequality that this type of interaction satisfies (informally speaking) Lipschitz continuity with respect to relative entropy\footnote{which is not a metric but appears to be  quite useful in our calculations}, i.e. \begin{equation}\label{entropy?}
\|B(t,x,\mu)-B(t,x,\nu)\|^2\leq C H(\mu\mid\nu),\end{equation} assuming that $\nu$ has subGaussian tails $e^{cx^2}\nu(dx)<M<\infty,$ see also \eqref{moment} and \eqref{7.6}.
 Therefore we may well replace the linear interaction $\langle \mu_t, b(t,x,\cdot)\rangle$ by a nonlinear one $B(t,x,\mu)$ satisfying \eqref{entropy?}, with $\nu$ satisfying a bound on its Gaussian moments.
In the third level of complexity, we assume $b$ only satisfies a linear growth condition, and nothing else. In this case it seems unlikely that we can find a relation like \ref{entropy?} that can hold uniformly in $x$, and we can at best hope for a relation like \ref{entropy?} to hold only after we take average over a weight function. This is the reason why we restrict ourselves to the case of linear action $\langle \mu_t, b(t,x,\cdot)\rangle.$ In the proof of Theorem \ref{theorem01}, we used the fact that the variable $x$ is given by the solution to an SDE, and we can use estimates of the law of that SDE to finish our (desired) averaging procedure. This is the most critical observation that leads us to improve \cite{lacker2021hierarchies}, Theorem 2.10.

In many cases of interest in analysis and PDE, it is easier to work on a bounded region, and we need to impose some extra weight functions when we switch to the whole space. A related example in the SPDE setting can be found in \cite{hairer2018multiplicative}. We make no claim that relative entropy is the only option when one wishes to impose some weight functions, indeed in the PDE literature other methods such as weighted $L^p$ norms are more frequently used, see for example \cite{jabin2021mean} and \cite{bresch2022new}. Perhaps one of the benefits  of relative entropy is that it is somewhat automatic: we are taking an average over some weight function without specifying which weight function we are actually using.

\end{remark}

\begin{remark}[The special case of an interaction kernel]\label{a special case}
Throughout the paper we are motivated by the following type of McKean-Vlasov SDEs 
\begin{equation}
   \label{lineargrowthd}
dX_t=\langle \mu_t, b(X_t,\cdot)\rangle dt+dW_t,\quad \operatorname{Law}(X_t)=\mu_t,\end{equation}
for some given function $b:\mathbb{R}^d\to\mathbb{R}^d\to\mathbb{R}^d$. If $b$ does not depend on its second component, we are reduced to the setting of SDEs $dX_t=b(X_t)dt+dW_t$ that are not density dependent. The following special case is often more interesting:
\begin{equation}\label{convolutionalcase}dX_t=\langle \mu_t, b(X_t-\cdot)\rangle dt+dW_t,\quad \operatorname{Law}(X_t)=\mu_t,\end{equation} for some kernel $b:\mathbb{R}^d\to\mathbb{R}^d$. Although \eqref{convolutionalcase} is nothing but a special case of \eqref{lineargrowthd}, the equation \eqref{convolutionalcase} has far more interesting properties due to its convolution structure. We refer to \cite{de2022multidimensional} for the recent well-posedness results of \eqref{convolutionalcase} that go beyond the scope of this paper, and to the author's recent work \cite{han2022smoothness} for a novel smoothing property of \eqref{convolutionalcase}.
\end{remark}

\begin{remark}[Nonlinear density dependence and linear functional derivative]Most of the results in this paper work in the case of a drift $B(t,x,\mu)$ which is Lipschitz continuous in its last argument in total variation distance. There is another notation in the literature that describes density dependence, which is called linear functional derivative, used intensively in for example \cite{de2021backward}. Existence of Linear functional derivative is a stronger assumption and often implies Lipschitz continuity of $B(t,x,\mu)$ in $\mu$, therefore we do not introduce linear functional derivatives in this paper. This notion becomes essential when we consider density dependent diffusion coefficients $\sigma(t,x,\mu)$, see \cite{de2021backward} and \cite{CHAUDRUDERAYNAL20221}.

\end{remark}

\begin{remark}[Range of application in the SPDE setting]

For SPDEs we have only worked in the case of one space dimension. This is because in higher dimensions, solutions to the stochastic heat equation (with space-time white noise) are no longer function-valued (see also \cite{karczewska2019note} for the case of stochastic wave equation), and we need to adapt our function spaces accordingly. If the noise is white in time but coloured in space, we could find classical solutions in higher dimensions (see for example \cite{dalang1998stochastic}), and our solution to McKean-Vlasov SPDEs can be extended to that case.

Moreover, for singular SPDEs that require a renormalization procedure (via regularity structures \cite{hairer2014theory}, paracontrolled calculus \cite{gubinelli2015paracontrolled}, etc.) to make sense of the solution, including the KPZ equation\cite{hairer2013solving} or the dynamical $\Phi_3^4$ model\cite{hairer2014theory}, our strategy clearly fails in this setting as we have no Girsanov transform at hand. The theories \cite{hairer2014theory}, \cite{gubinelli2015paracontrolled} give solutions in a pathwise sense, providing no clear probabilistic picture that allows us to do a change of measure. We do mention the work \cite{shen2022large} on mean field behavior related to the dynamical $\Phi_2^4$ model.
\end{remark}

\begin{remark}[Path dependent drifts and singular drifts]\label{fywanghx} Professor Xing Huang and Fengyu Wang kindly pointed out that in the setting of Theorem \ref{lplqfirst} and \ref{wes}, even though the singular part $b_1$ is only state dependent, we can take the locally bounded parts $b_0$ and $b_2$ to be path dependent, just as in the situation of Theorem \ref{theorem01}. As we can always measure relative entropy on the whole path space $\mathcal{C}([0,T];\mathbb{R}^d)$, whether a drift is state dependent or path dependent does not make that much a difference in our estimate. Though the extension is obvious, we keep the state dependent notations to simplify the presentation. 
\end{remark}

\begin{remark}
[Singular drifts: the critical case]
The condition $b(t,x)\in L_t^q([0,T];L_x^p(\mathbb{R}^d))$ with $\frac{d}{p}+\frac{2}{q}<1$ is the one assumed in Krylov-Röckner \cite{krylov2005strong} for the SDE 
$dX_t=b(t,X_t)dt+dW_t$ to have a unique strong solution. Theorems \ref{lplqfirst} and \ref{wes} are both set under this assumption.

The critical case $\frac{d}{p}+\frac{2}{q}=1$ is left open for quite a while, and is settled only very recently in a series of papers \cite{krylov2021stochastic}, \cite{rockner2020sdes}, \cite{rockner2021sdes}. It is tempting to extend our theorem \ref{lplqfirst} to this critical case, but the Krylov type estimate provided in \cite{krylov2021stochastic} is not suitable enough for us to carry out the same proof. Thus we leave this question open for future study. We do mention a related work \cite{zhang2020weak} that derives weak existence results under a much more general condition.
\end{remark}

\section{ Linear growth interactions, proof of Theorem \ref{theorem01}}\label{chapter2}
Consider the SDE $$dX_t=c(t,X)dt+dW_t\quad X_0=x_0,$$
where $c:[0,T]\times\mathcal{C}_T^d\to\mathbb{R}^d$ is progressively measurable and has linear growth: \begin{equation}\label{lineargrowth}|c(t,x)|\leq K(1+
\|x\|_t)\quad\text{ for some }K>0,\text{ for all }(t,x)\in [0,T]\times\mathcal{C}_T^d.\end{equation} Then the SDE has a unique weak solution. The weak well-posedness of this SDE is a consequence of the linear growth condition and Girsanov's theorem, see for example \cite{karatzas1998methods} and \cite{liptser2013statistics}.

The following proposition, already proved in \eqref{mainproof}, is the building block of our argument. 
\begin{proposition}[Proposition 7.1 of \cite{lacker2021hierarchies}] \label{lemma8}Assume $b_0$ and $b$ are bounded and progressively measurable. Then the McKean-Vlasov SDE
$$
d X_{t}=\left(b_{0}(t, X)+\langle\mu, b(t, X, \cdot)\rangle\right) d t+d W_{t}, \quad \mu=\operatorname{Law}(X)
$$
has a unique in law weak solution from any initial distribution.
\end{proposition}
Proposition 7.1 of \cite{lacker2021hierarchies} is stated for a much wider family of drifts $b_0$. In particular, $b_0$ can have linear growth, i.e., \eqref{lineargrowth} is satisfied with $b_0$ in place of $c$. We will use this fact in the proof of Theorem \ref{theorem01}. Alternatively, we may also set the drift to be $0$ and replace $b(t,x,y)$ by $b(t,x,y)+b_0(x)$, then the truncation argument in the proof works without change.

The following weighted inequality is useful in mean field problems, see also \cite{JABIN20163588}.
\begin{lemma}[Weighted Pinsker inequality, Theorem 2.1 of \cite{AFST_2005_6_14_3_331_0} and equation (6.1) of \cite{lacker2021hierarchies}]
For probability measures $\nu$ and $\nu^{\prime}$ on a common measurable space, and for a measurable $\mathbb{R}^{d}$-valued function $f$,
\begin{equation}\label{pinsker}
\left|\left\langle\nu-\nu^{\prime}, f\right\rangle\right|^{2} \leq 2\left(1+\log \int e^{|f|^{2}} d \nu^{\prime}\right) H\left(\nu \mid \nu^{\prime}\right).
\end{equation}
\end{lemma}

Now we prove Theorem  \ref{theorem01}.
\begin{proof}
Uniqueness. We first prove there is at most one weak solution satisfying $\mathbb{E}\left\|X\right\|_T<\infty$. Assume given $X^{i} \sim \mu^{i}$, $i=1,2$ are two solutions,
$$
d X_{t}^{i}=\left(b_{0}\left(t, X^{i}\right)+\left\langle\mu^{i}, b\left(t, X^{i}, \cdot\right)\right\rangle\right) d t+d W_{t}^{i}, \quad X_{0}^{i} \sim \mu_{0}, \quad \mu^{i}=\operatorname{Law}\left(X^{i}\right).
$$
Under the assumption \eqref{linear}, for $t \in[0, T]$ we have
$$
\left\|X^{i}\right\|_{t} \leq\left|X_{0}^{i}\right|+K T+K \int_{0}^{t}\left(\left\|X^{i}\right\|_{s}+\mathbb{E}\left\|X^{i}\right\|_{s}\right) d s+\left\|W^{i}\right\|_{t}.
$$
The first step is to estimate the expectation $\mathbb{E}\left\|X^i\right\|_T.$ We take expectations on both sides, noting that $\mathbb{E}\|X^i\|_T<\infty.$ Then we apply Gronwall and get $$\mathbb{E}\left\|X^{i}\right\|_{T} \leq e^{2 K T}\left(\mathbb{E}\left|X_{0}^{i}\right|+4 d T\right),$$ where we used also $\mathbb{E}\left\|W^{i}\right\|_{T}^{2} \leq 4 \mathbb{E}\left|W_{T}^{i}\right|^{2}=4 d T$. The next step is to obtain exponential moment estimates. We use Gronwall to find that 
\begin{equation}\label{ito}
\left\|X^{i}\right\|_{T} \leq e^{K T}\left(\left|X_{0}^{i}\right|+K T+K T \mathbb{E}\left\|X^{i}\right\|_{T}+\left\|W^{i}\right\|_{T}\right).
\end{equation}
Since $\mathbb{E} e^{c\|W\|_{T}^{2}}<\infty$ for some $c>0$, by the exponential moment assumption on $\mu_{0},$ we may find $c>0$ small enough and $C<\infty$ such that 
for $i=1,2$, 
\begin{equation}\label{moment}
\int_{\mathcal{C}_{T}^{d}} e^{c\|x\|_{T}^{2}} \mu^{i}(d x)=\mathbb{E} e^{c\left\|X^{i}\right\|_{T}^{2}} \leq C.
\end{equation}

 An application of weighted Pinsker inequality \eqref{pinsker} reveals
\begin{equation}\label{7.6}
\begin{aligned}
H\left(\mu^{1}[t] \mid \mu^{2}[t]\right) &=\frac{1}{2} \int_{\mathcal{C}_{T}^{d}} \int_{0}^{t}\left|\left\langle\mu^{1}-\mu^{2}, b(s, x, \cdot)\right\rangle\right|^{2} d s \mu^1(dx) \\
& \leq \epsilon^{-1}\left(1+R_{\epsilon}(\mu^1,\mu^2)\right) \int_{0}^{t} H\left(\mu^{1}[s] \mid \mu^{2}[s]\right) d s,
\end{aligned}
\end{equation}
where for each $\epsilon>0$, for each pair $\left(\mu,\nu\right)\in\mathcal{P}(\mathcal{C}_T)$, we denote
\begin{equation}
R_{\epsilon}(\mu,\nu):=\sup _{t \in[0, T]} \log \int \exp \left(\epsilon\left|b(t, x, y)\right|^{2}\right) \mu(dx)\nu(d y).\label{Repsilon}
\end{equation}
If we choose $\epsilon$ sufficiently small, then from \eqref{moment} and \eqref{linear} we deduce that \begin{equation}R_{\epsilon}(\mu^1,\mu^2)<\infty.\label{uniqueproof}\end{equation} Thus,  Gronwall's inequality applied to \eqref{7.6} implies that $H\left(\mu^{1}[t] \mid \mu^{2}[t]\right)=0$ for all $t$, and so $\mu^{1}=\mu^{2} .$

Existence. The strategy is to truncate the interaction $b$ and solve the truncated version using Proposition \ref{lemma8}. The harder part is to take limits properly.

Step 1. Define $b^{n}:=(b \wedge n) \vee(-n)$ for each $n \in \mathbb{N}$. ($b$ is  $d$-dimensional, so we take min and max in each coordinate.) Using Proposition \ref{lemma8} and the words after it, the boundedness of $b^{n}$ implies we can find a unique solution $X^{n} \sim \mu^{n}$ of the McKean-Vlasov equation
$$
d X_{t}^{n}=\left(b_{0}\left(t, X^{n}\right)+\left\langle\mu^{n}, b^{n}\left(t, X^{n}, \cdot\right)\right\rangle\right) d t+d W_{t}^{n}, \quad X_{0}^{n} \sim \mu_{0}, \quad \mu^{n}=\operatorname{Law}\left(X^{n}\right).
$$
Since $b^{n}$ is a truncated version, it should satisfy the same linear growth condition \eqref{linear}. Then we have as in \eqref{moment} that there exist $c, C_{1}>0$ such that
\begin{equation}\label{momentn}
\sup _{n} \int_{\mathcal{C}_{T}^{d}} e^{c\|x\|_{T}^{2}} \mu^{n}(d x) \leq C_{1}.
\end{equation}
This estimate is uniform in $n$, which implies that if we choose $\epsilon$ sufficiently small, 
\begin{equation}\label{uniformestimate}
R_\epsilon:=\sup_{n,m} R_\epsilon(\mu^n,\mu^m)<\infty.\end{equation}

Step 2. The next step is to prove $\left(\mu^{n}\right)_{n \in \mathbb{N}}$ converges in a certain sense. We have for each $n, m \in \mathbb{N}$ and $t \in[0, T],$
\begin{equation}\label{123321}
\begin{aligned}
H\left(\mu^{n}[t] \mid \mu^{m}[t]\right)&= \frac{1}{2} \int_{C_{T}^{d}} \int_{0}^{t}\left|\left\langle\mu^{n}, b^{n}(s, x, \cdot)\right\rangle-\left\langle\mu^{m}, b^{m}(s, x, \cdot)\right\rangle\right|^{2} d s \mu^{n}(d x) \\
&\le  \frac{3}{2} \int_{C_{T}^{d}} \int_{0}^{t}\left|\left\langle\mu^{n}, b^{n}(s, x, \cdot)-b(s, x, \cdot)\right\rangle\right|^{2} d s \mu^{n}(d x) \\
&\quad+\frac{3}{2} \int_{C_{T}^{d}} \int_{0}^{t}\left|\left\langle\mu^{m}, b^{m}(s, x, \cdot)-b(s, x, \cdot)\right\rangle\right|^{2} d s \mu^{n}(d x) \\
&\quad+\frac{3}{2} \int_{C_{T}^{d}} \int_{0}^{t}\left|\left\langle\mu^{n}-\mu^{m}, b(s, x, \cdot)\right\rangle\right|^{2} d s \mu^{n}(d x).
\end{aligned}
\end{equation}
By Jensen's inequality,
$$
\begin{aligned}
&\int_{\mathcal{C}_{T}^{d}}\left|\left\langle\mu^{n}, b^{n}(s, x, \cdot)-b(s, x, \cdot)\right\rangle\right|^{2} \mu^{n}(d x) \\
&\quad \leq \int_{\mathcal{C}_{T}^{d}} \int_{\mathcal{C}_{T}^{d}}|b(s, x, y)|^{2} 1_{\{|b(s, x, y)| \geq n\}} \mu^{n}(d x) \mu^{n}(d y).
\end{aligned}
$$

By the uniform estimate \eqref{momentn} and \eqref{linear}, the right-hand side goes to zero as $n \rightarrow \infty$. Similarly,
$$
\begin{aligned}
\int_{\mathcal{C}_{T}^{d}} &\left|\left\langle\mu^{m}, b^{m}(s, x, \cdot)-b(s, x, \cdot)\right\rangle\right|^{2} \mu^{n}(d x) \\
& \leq \int_{\mathcal{C}_{T}^{d}} \int_{\mathcal{C}_{T}^{d}}|b(s, x, y)|^{2} 1_{\{|b(s, x, y)| \geq m\}} \mu^{m}(d y) \mu^{n}(d x)
\end{aligned}
$$
vanishes as $n, m \rightarrow \infty$ for the same reasons.

By the weighted Pinsker inequality \eqref{pinsker} and estimate \eqref{uniformestimate}, for some sufficiently small $\epsilon>0$,
$$
\begin{aligned}
&\int_{C_{T}^{d}} \int_{0}^{t}\left|\left\langle\mu^{n}-\mu^{m}, b(s, x, \cdot)\right\rangle\right|^{2} d t \mu^{n}(d x)\\&\quad \leq \frac{2}{\epsilon}(1+R_\epsilon(\mu^n,\mu^m)) \int_{0}^{t} H\left(\mu^{n}[s] \mid \mu^{m}[s]\right) d s\\&\quad\leq \frac{2}{\epsilon}(1+R_\epsilon) \int_{0}^{t} H\left(\mu^{n}[s] \mid \mu^{m}[s]\right) d s .
\end{aligned}
$$
By Gronwall, we deduce that $$\lim _{n, m \rightarrow \infty} H\left(\mu^{n}[t] \mid \mu^{m}[t]\right)=0.$$ The claim that $\left(\mu^{n}\right)_{n \in \mathbb{N}}$ is a $\|\cdot\|_{\mathrm{TV}}$-Cauchy sequence follows from Pinsker's inequality. Then $\left\|\mu^{n}-\mu\right\|_{\mathrm{TV}} \rightarrow 0$ for some $\mu \in \mathcal{P}\left(\mathcal{C}_{T}^{d}\right)$.  

Step 3. Since each $\mu^n$ satisfies \eqref{momentn} and $\mu^n$ converges to $\mu$ in total variation, $\mu$ has a finite first moment $\mathbb{E}_\mu[\|X\|_T]<\infty.$

The convergence actually holds for test functions with linear growth. If $f: \mathcal{C}_{T}^{d} \rightarrow \mathbb{R}$ is measurable with $\sup _{x}|f(x)| /\left(1+\|x\|_{T}\right)<\infty$, then $\left\langle\mu^{n}, f\right\rangle \rightarrow\langle\mu, f\rangle .$ 
The strategy is to truncate $f$ as $f1_{\{|f|\leq k\}}$ and $f1_{\{|f|>k\}}$, and note that for $k>0$,
$$
\int_{\{|f|>k\}}|f| d \mu^{n} \leq \frac{1}{k}\int |f|^2 d \mu^{n}.
$$ The right hand side vanishes as $k\to\infty$ uniform in $n$ because of the (uniform in $n$) exponential moment estimate \eqref{momentn}.

 Writing
$$
\left\langle\mu^{n}-\mu, f\right\rangle=\int_{\{|f| \leq k\}} f d\left(\mu^{n}-\mu\right)+\int_{\{|f|>k\}} f d\left(\mu^{n}-\mu\right),
$$
the right hand side vanishes by sending $n \rightarrow \infty$ and then $k \rightarrow \infty$. As $\mu$ has a finite first moment, $\int_{\{|f|>k\}} f d \mu \rightarrow 0$ as $k \rightarrow \infty$. This completes the proof of step 3.

Step 4. We complete the proof by showing that $\mu$ is the law of a solution to the McKean-Vlasov SDE. 

Since $\mu$ has a finite first moment, the function $\langle b(t,x,\cdot),\mu\rangle$ has  linear growth in $x$. Define $\bar{b}_\mu(t,x):=b_{0}\left(t, x\right)+\left\langle\mu, b\left(t, x, \cdot\right)\right\rangle$, then the SDE $$dX_t=\bar{b}_\mu(t,X_t)dt+dW_t$$ 
admits a unique weak solution with initial distribution $\mu_{0}$. We denote its law on $\mathcal{C}_T^d$ by $\widetilde{\mu}$. We wish to show $\widetilde{\mu}=\mu$. A similar estimate as in the uniqueness part shows there exists $c'$, $C_1'>0$ such that 
\begin{equation}\label{momentbar}
 \int_{\mathcal{C}_{T}^{d}} e^{c'\|x\|_{T}^{2}} \widetilde{\mu}(d x) \leq C_{1}'.
\end{equation}

We claim first that, uniformly in $t \in [0,T]$,
\begin{equation}\label{7.11}
\lim _{n \rightarrow \infty} \int_{\mathcal{C}_{T}^{d}}\left|\left\langle\mu^{n}, b^{n}(t, x, \cdot)\right\rangle-\langle\mu, b(t, x, \cdot)\rangle\right|^2 \widetilde{\mu}(d x) = 0.
\end{equation}
To see this, first note the definition of $b^{n}$ and Jensen's inequality imply that
$$\begin{aligned}
&\int_{\mathcal{C}_{T}^{d}}\left|\left\langle\mu^{n}, b^{n}(t, x, \cdot)-b(t, x, \cdot)\right\rangle\right|^2 \widetilde{\mu}(d x)\\&\quad \leq \int_{\mathcal{C}_{T}^{d}} \int_{\mathcal{C}_{T}^{d}}|b(t, x, y)|^2 1_{\{|b(t, x, y)| \geq n\}} \widetilde{\mu}(d x) \mu^{n}(d y)\end{aligned}
$$
which converges to zero thanks to the linear growth assumption \eqref{linear}, \eqref{momentn} and \eqref{momentbar}. Hence, to prove \eqref{7.11}, it suffices to show that, uniformly in $t\in[0,T]$, $$\lim_{n\to\infty}\int_{\mathcal{C}_T^d}\left|a_{n}(t,x)\right|^2\widetilde{\mu}(dx)= 0,$$ where we set $a_{n}(t,x):=\left\langle\mu^{n}-\mu, b(t, x, \cdot)\right\rangle .$ The result of Step 3 and the linear growth assumption on $b$ imply that $a_{n} \rightarrow 0$ pointwise. Using linear growth property \eqref{linear}, we have
$$
\begin{aligned}
&\left|a_{n}(t,x)\right|  \leq \int_{\mathcal{C}_{T}^{d}} \int_{\mathcal{C}_{T}^{d}}\left|b(t, x, y)-b\left(t, x, y^{\prime}\right)\right| \mu^{n}(d y) \mu\left(d y^{\prime}\right) \\
&\quad \leq K\left(2+2\|x\|_T+\int_{\mathcal{C}_{T}^{d}}\|y\|_{T} \mu^{n}(d y)+\int_{\mathcal{C}_{T}^{d}}\left\|y^{\prime}\right\|_{T} \mu\left(d y^{\prime}\right)\right).
\end{aligned}
$$

By the dominated convergence theorem and \eqref{momentbar}, we may send $n \rightarrow \infty$ to find $\left\langle\widetilde{\mu},\left|a_{n}(t,\cdot)\right|^2\right\rangle \rightarrow 0$ uniformly on $[0,T]$. This proves equation \eqref{7.11}.

We compute via Girsanov transform that
\begin{equation}\label{7.21}
H\left(\widetilde{\mu}[T]\mid \mu^n[T]\right)= \frac{1}{2} \int_{\mathcal{C}_{T}^{d}}\int_0^T\left|\left\langle\mu^{n}, b^{n}(t, x, \cdot)\right\rangle-\langle\mu, b(t, x, \cdot)\rangle\right|^2dt \widetilde{\mu}(d x),\end{equation}
the right hand side converges to $0$ as $n\to\infty$ thanks to \eqref{7.11}.

Then Pinsker's inequality implies $\mu^n$ converges to $\widetilde{\mu}$ in total variation distance. But $\mu^n$ also converges to $\mu$ in total variation, which implies $\mu=\widetilde{\mu}$. 

This proves the McKean-Vlasov SDE is well-posed in weak probabilistic sense, and the law $\mu$ of its solution has been constructed.
\end{proof}

\section{Interactions with Krylov's integrability condition}\label{chapter3}
\subsection{Krylov's estimate}
We first introduce the following Krylov's estimate and Khasminskii’s estimate. They are both standard and can be found in \cite{krylov2005strong}.

\begin{proposition}[Khasminskii’s estimate]
For any $\lambda, T>0$ and  $f\in L_t^{q}([0,T],L_x^{p}(\mathbb{R}^d))$, $(p,q)\in\mathscr{I}_1$,
\begin{equation}\label{kraim}
\mathbb{E} \exp \left(\lambda \int_{0}^{T}\left|f_{s}\left(W_{s}\right)\right|^2 \mathrm{d} s\right) \leqslant C_{3},
\end{equation}
where $W$ is the canonical Brownian motion on $\mathcal{C}_T^d$ and $C_{3}$ only depends on $\lambda, p, q, T$ and $\|f\| _{ L_t^{q}\left([0,T],L_x^{p}(\mathbb{R}^d)\right)} .$
\end{proposition}

\begin{proposition}[Krylov's estimate]\label{ndfdsgs}
Assume $W$ is the $d$-dimensional Brownian motion, and define $X_s=X_0+W_s$ where $X_0$ is a random variable independent of $W$, 
then for any $\left(p, q\right) \in \mathscr{I}_{1}$ and $T>0$, there is a universal constant $C>0$ such that  
\begin{equation}
\label{lplqcase}
\mathbb{E}\left(\int_{0}^{t} |f_{s}\left(X_{s}\right)|^2 \mathrm{d} s \mid \mathscr{F}_{0}\right) \leqslant Ct^{\frac{q-1}{q}-\frac{d}{2p}}\|f\|^2_{ L_t^{q}\left([0,T],L_x^{p}(\mathbb{R}^d)\right)}, \quad t\in[0,T].
\end{equation}
\end{proposition}

\begin{proof}
In the case of deterministic initial value $X_0$, that is when $\mathcal{F}_0$ is trivial, the estimate can be found in the first few lines of the proof to \cite{krylov2005strong}, Lemma 3.2. Since the estimate is independent of $X_0$, the case of a general initial law $X_0$ immediately follows.
\end{proof}

The following proposition is a generalization of Theorem 2.1 of \cite{krylov2005strong}. The strategy of proof is very similar.
\begin{proposition}\label{propopopo}
Consider the following  SDE
$$ dX_t=b_0(t,X_t)dt+h(t,X_t)dt+dW_t\quad \operatorname{Law}(X_0)=\mu_0.$$
 Assume $W$ is independent of $X_0$ and the following conditions are satisfied: \begin{itemize}
    \item For some $(p_1,q_1)\in\mathscr{I}_1$,
    $$h(t,x)\in L_t^{q_1}\left([0,T],L_x^{p_1}(\mathbb{R}^d)\right).$$
    \item For some $\beta\in[0,1)$ and $K>0$, $$|b_0(t,x)|\leq K(1+|x|^\beta).$$
    \item The initial law $\mu_0\in\mathcal{P}(\mathbb{R}^d)$satisfies$$\int_{\mathbb{R}^d} e^{\kappa |x|^{2\beta}}\mu_0(dx)<\infty\quad \text{ for all } \kappa\in\mathbb{R}.$$
\end{itemize}
Denote by $\mathbb{Q}$ the law of $X_0+W$ on $\mathcal{C}_T^d$. Then

 (i) This SDE admits a unique weak solution on $[0,T]$ satisfying 
 \begin{equation}\label{?!}\int_0^T \left|b_0(t,X_t)+h(t,X_t)\right|^2dt<\infty\text{ with probability 1 }.\end{equation}

(ii) Denote by $\mathbb{P}$ the law of this SDE on $\mathcal{C}_T^d$, then for every  $\kappa\geq1$, there exists $M(\kappa)<\infty$ depending only on $\kappa$, $\mu_0$, $K$, $\beta$ and $\|h\|_{L_t^{q_1}\left([0,T],L_x^{p_1}(\mathbb{R}^d)\right)}$ such that
\begin{equation}\label{lplpequivmeasure}\mathbb{E}_\mathbb{Q}\left[\left(\frac{d\mathbb{Q}}{d\mathbb{P}}\right)^\kappa+\left(\frac{d\mathbb{P}}{d\mathbb{Q}}\right)^\kappa\right]<M(\kappa)<\infty,\quad \kappa\geq 1.\end{equation}

(iii) Krylov's estimate applies to $X$, i.e., for any $f\in L_t^{q_1}\left([0,T],L_x^{p_1}(\mathbb{R}^d)\right)$,
\begin{equation}
\label{krylovforx}
\mathbb{E}\left(\int_{0}^{t} |f_{s}\left(X_{s}\right)|^2 \mathrm{d} s \mid \mathscr{F}_{0}\right) \leqslant Ct^{\frac{q_1-1}{q_1}-\frac{d}{2p_1}}\|f\|^2_{ L_t^{q_1}\left([0,T],L_x^{p_1}(\mathbb{R}^d)\right)} .
\end{equation}

The constant $C$ depends on $\mu_0$, $K$, $\beta$, $p_1$, $q_1$ and $\|h\|^2_{ L_t^{q_1}\left([0,T],L_x^{p_1}(\mathbb{R}^d)\right)}$.
\end{proposition}

\begin{proof}
Define $$\rho(W):=\exp \left(\int_{0}^{T} \left(b_0+h\right)(t,X_0+W_t) d W_{t}-\frac{1}{2} \int_{0}^{T}\left|(b_0+h)(t,X_0+W_t)\right|^{2} d t\right).$$

Weak existence of the SDE follows from Girsanov transform, where Novikov's condition can be checked via justifying
$$\mathbb{E}\left[\rho(W)^\kappa\right]<\infty \text{ for all } \kappa>0.$$ This is a consequence of \eqref{kraim} and Fernique's theorem: note that $\beta\in[0,1)$, then
\begin{equation}\label{edc}\mathbb{E}\left[\kappa\int_0^T \left|W_s\right|^{2\beta}ds\right]<\kappa(\beta)<\infty\quad\text{ for all } \kappa>0.\end{equation}

The same reasoning justifies \eqref{lplpequivmeasure} for the solution $X$ we constructed via Girsanov transform.

For weak uniqueness, since \eqref{?!} is also satisfied with the Brownian motion $W$ in place of $X$, condition \eqref{?!} implies that we can apply Liptser-Shiryaev's theorem about absolutely continuous change of measure (see also Lemma 4.4 of \cite{lacker2021hierarchies}). Assume that $X$ is a solution to this SDE and satisfies \eqref{?!}, then for any Borel nonnegative function $f$ defined on the space $C\left([0, T), \mathbb{R}^{d}\right)$, we have
$$
E f(X .)=E f(W .) \rho(W_\cdot).
$$

Since all polynomial moments of $\rho$ are finite, Krylov's estimate \eqref{krylovforx} for this solution $X$ can be obtained via Cauchy-Schwartz and the Krylov estimate for Brownian motion $W$ \eqref{lplqcase}, see for example proof of Lemma 3.3 in \cite{krylov2005strong}. Then weak uniqueness is now a consequence of Girsanov's theorem, since the Girsanov density of each weak solution has the same form. 

\end{proof}

For $h(t,x)\in L_t^{q_1}\left([0,T],L_x^{p_1}(\mathbb{R}^d)\right)$ with $(p_1,q_1)\in\mathscr{I}_1$, and for $\mu^i\in \mathcal{P}(\mathcal{C}_T^d)$, $i=1,2$, Jensen's inequality implies that there exists some $\ell_t\in L^{q_1}([0,T]$)  such that we have the following estimate
\begin{equation}\label{zxcvbn}
    \|\left\langle \mu_t^1-\mu_t^2, h(t,x-\cdot)\right\rangle\|^2_{L_t^{q_1}\left([0,T],L_x^{p_1}(\mathbb{R}^d)\right)}\leq \left( \int_0^T \ell^{q_1}_t \|\mu_t^1-\mu_t^2\|_{TV}^{q_1}dt\right)^{2/q_1}.
\end{equation}
The function $\ell_t$ depends only on $\|h\|_{L_t^{q_1}\left([0,T],L_x^{p_1}(\mathbb{R}^d)\right)}.$

Another application of Jensen's inequality implies 
\begin{equation}\label{independence}
    \|\left\langle \mu_t^1, h(t,x-\cdot)\right\rangle\|^2_{L_t^{q_1}\left([0,T],L_x^{p_1}(\mathbb{R}^d)\right)}\leq\|h\|^2_{L_t^{q_1}\left([0,T],L_x^{p_1}(\mathbb{R}^d)\right)}.
\end{equation}

\subsection{Proof of Theorem \ref{lplqfirst}}

\begin{proof}
For any $\mu\in\mathcal{P}(\mathcal{C}_T^d)$ define $$\bar{b}_{\mu}(t,x):=b_0(t,x)+\left\langle \mu_t, b_1(t,x,\cdot)\right\rangle+\left\langle \mu_t, b_2(t,x,\cdot)\right\rangle.$$
Then by Proposition \ref{propopopo}, the SDE $$dX_t=\bar{b}_\mu(t,X_t)dt+dW_t,\quad \operatorname{Law}(X_0)=\mu_0$$ admits a unique weak solution. We denote by $\Phi(\mu)$ the law of this solution.

By Proposition \ref{propopopo}, $X_t$ satisfies the Krylov estimate 
\begin{equation}\label{asdfgh}
\mathbb{E}\left(\int_{0}^{t} |f_{s}\left(X_{s}\right)|^2 \mathrm{d} s \mid \mathscr{F}_{0}\right) \leqslant C_{T,d,q_1,p_1}\|f\|^2_{{L_t^{q_1}\left([0,T],L_x^{p_1}(\mathbb{R}^d)\right)}}.\end{equation}
By \eqref{independence}, we may choose the constant independently of $\mu\in\mathcal{P}(\mathcal{C}_T^d).$

 Then for any $\mu,\nu\in\mathcal{P}(\mathcal{C}_T^d)$, using the inequality $(a+b)^2\leq 2a^2+2b^2$ and \eqref{zxcvbn},

 \begin{equation}\begin{aligned}
    \|&\Phi(\mu)_t-\Phi(\nu)_t\|^2_{TV}\leqslant 2 H\left(\Phi(\mu)[t]\mid\Phi(\nu)[t]\right)\\&\leq 2\mathbb{E}\int_0^t\left(\left|\left\langle \mu_s-\nu_s, b_1\left(s,X_s,\cdot\right)\right\rangle\right|^2+\left|\left\langle \mu_s-\nu_s, b_2\left(s,X_s,\cdot\right)\right\rangle
    \right|^2\right)ds\\&\leqslant 2\int_0^t\|b_2\|^2_\infty \|\mu_s-\nu_s\|^2_{TV}ds+ 2C\left(\int_0^t \ell^{q_1}_s \|\mu_s-\nu_s\|^{q_1}_{TV} ds\right)^{2/q_1}.
    \end{aligned}
\end{equation}
The first inequality follows from Pinsker's inequality and \eqref{processing}. In the last line, the constant $C$ comes from the constant given in \eqref{asdfgh} and we have replaced the $L_t^{q_1}-L_x^{p_1}$ norm by a function $\ell_t\in L^{q_1}([0,T])$ as a consequence of Jensen's inequality \eqref{zxcvbn} for $b_1$.
The estimate follows as we use Pinsker's inequality for $b_2$ and  Krylov's estimate \eqref{asdfgh}  for $b_1$. The function $\ell_t\in L^{q_1}([0,T])$ is independent of the choice of $\mu$ and $\nu$.

Since $q_1> 2$ and $\ell_t^{q_1}\in L^1([0,T])$, a standard Picard iteration finishes the proof.
\end{proof}

\begin{remark} \label{twolplp}
Fix two pairs 
$(p_1,q_1)\in\mathscr{I}_1$ and $(p_2,q_2)\in\mathscr{I}_1$, we may find two functions 
$$h^1\in L_t^{q_1}\left([0,T],L_x^{p_1}(\mathbb{R}^d)\right)
\quad  h^2\in L_t^{q_2}\left([0,T],L_x^{p_2}(\mathbb{R}^d)\right).$$

Then a straightforward generalization of our proof  allows us to solve 
$$dX_t=\left\langle \mu_t,\left(h_t^1+h_t^2\right)(X_t-\cdot)\right\rangle dt+dW_t,\quad\operatorname{Law}(X_t)=\mu_t.$$
To our best knowledge, \eqref{mckean11} with this type of mixed interactions is new in the literature
\end{remark}

\subsection{Proof of Theorem \ref{wes} }\begin{proof}
Existence. Step 1. 
We will truncate $b_2$. Define $b_2^{n}:=(b_2 \wedge n) \vee(-n)$ for each $n \in \mathbb{N}$. (We take min and max coordinate by coordinate.)

By Theorem \ref{lplqfirst} and boundedness of $b_2^{n}$, there exists a unique solution $X^{n} \sim \mu^{n}$ of the McKean-Vlasov equation
$$
d X_{t}^{n}=\left\langle\mu_t^{n}, b_1+b_2^{n}\left(t, X_t^{n}, \cdot\right)\right\rangle d t+d W_{t}^{n}, \quad X_{0}^{n} \sim \mu_{0}, \quad \mu^{n}=\operatorname{Law}\left(X^{n}\right).
$$

By the assumption we set on $b_2$, $$\left|\left\langle \mu_t^n,b_2^n(t,x,\cdot)\right\rangle\right|\leq K(1+|x|^\beta)\quad\text{ for each }n.$$
It follows that $X^n$ satisfies the conditions of Proposition \ref{propopopo} for each $n$, with all the constants uniform in $n$. Consequently we have a uniform in $n$ Krylov estimate as in \eqref{asdfgh} with $X_s^n$ in place of $X_s$. Since \eqref{lplpequivmeasure} holds for the law of $X^n$ with constant independent of $n$, we also have a uniform in $n$ exponential moment estimate: for some $c>0$, $$\int_{\mathcal{C}_T^d}e^{c\|x\|_T^2}\mu_n(dx)<C<\infty.$$
Define
\begin{equation}
R_{\epsilon}(\mu,\nu):=\sup _{t \in[0, T]} \log \int \exp \left(\epsilon\left|b_2(t, x, y)\right|^{2}\right) \mu(dx)\nu(d y).\label{Repsilon12}
\end{equation}
Since $b_2$ has sublinear growth, we may find some $\epsilon>0$ small enough that

\begin{equation}\label{uniformestimate12}
R_\epsilon:=\sup_{n,m} R_\epsilon(\mu^n,\mu^m)<\infty.\end{equation}

Step 2. We replace $b^n$ by $b_1+b_2^n$ in equation \eqref{123321}. For terms involving $b_2^n$, we proceed as in step 2. For terms involving $b_1$, use Krylov's estimate and \eqref{zxcvbn}. Finally use Pinsker's inequality to bound total variation by relative entropy. We arrive at the same conclusion
$$\lim _{n, m \rightarrow \infty} H\left(\mu^{n}[t] \mid \mu^{m}[t]\right)=0.$$

Step 3. There exists $\mu\in\mathcal{P}(\mathcal{C}_T^d)$ such that $\mu^n$ converges to $\mu$ in total variation. Assume $f$ has linear growth, then $\left\langle \mu_n,f\right\rangle\to \left\langle \mu,f\right\rangle.$ This proof is identical to step 3 of Theorem \ref{theorem01}. 

Step 4. We follow the same lines as in Theorem \ref{theorem01}, step 4. By proposition \ref{propopopo}, the following SDE (with $\mu$ given in Step 3) $$
d X_{t}=\langle\mu_t, b_1+b_2(t, X_t, \cdot)\rangle d t+dW_t, \quad \mu_0=\operatorname{Law}(X_0)
$$ has a unique weak solution, whose law on $\mathcal{C}_T^d$ we denote by $\widetilde{\mu}$. To prove $\mu=\widetilde{\mu}$, by Girsanov transform \eqref{Browniancase} and Pinsker's inequality it suffices to show that, denote by $b^n:=b_1+b_2^n$,
$$
\lim _{n \rightarrow \infty} \int_{\mathcal{C}_{T}^{d}}\int_0^T\left|\left\langle\mu^{n}, b^{n}(t, x, \cdot)\right\rangle-\langle\mu, b(t, x, \cdot)\rangle\right|^2dt \widetilde{\mu}(d x) = 0.
$$

Since $b_2$ has sublinear growth and $\mu^n$ has (uniform in $n$) squared exponential moments, the same reasoning in proof of Theorem \ref{theorem01}, Step 4 shows that we only need to establish 
$$
\lim _{n \rightarrow \infty} \int_{\mathcal{C}_{T}^{d}}\int_0^T\left|\left\langle\mu^{n}-\mu, b_1+b_2(t, x, \cdot)\right\rangle\right|^2dt \widetilde{\mu}(d x) = 0.
$$
For the term involving $b_2$, we follow the proof in Theorem \ref{theorem01} , step 4. For the term involving $b_1$, this is a direct consequence of Krylov's estimate \eqref{asdfgh}, Jensen's inequality \eqref{zxcvbn} and the fact that $\mu_n$ converges to $\mu$ in total variation. 

For weak uniqueness, we follow the proof of Theorem \ref{theorem01} , and use \eqref{asdfgh},  \eqref{zxcvbn} in equation \eqref{7.6} for the term involving $b_1$.\end{proof}
\begin{remark}\label{remark3}
The condition on $b_2$ in Theorem \ref{wes}, inspired by Mishura et al. \cite{mishura2020existence}, is restrictive compared to the linear growth assumption in Theorem \ref{theorem01}. We are not able to prove well-posedness of McKean-Vlasov SDEs with more general conditions on $b_2$ assuming the $b_1$ component is present. The reason is that if we only have a component of linear growth, we can use the a priori estimates given in the uniqueness part of Theorem \ref{theorem01}, but that estimate fails once we have a singular part $b_1$. One can also start from the singular part $b_1$, and use Krylov type estimates for that component, but this would force us to assume the stringent condition on $b_2$ because otherwise (as we have density dependence and we only have the estimate $|b_2(t,x,y)|\leq K(1+|x|+|y|)$ available) we need to obtain another a priori estimate on $\mathbb{E}\|X_t\|$ before we can get a priori estimates on $X$. That explains why we cannot easily construct a solution to the McKean-Vlasov SDE.

This problem does not appear in the SDE setting. It is proved in \cite{krylov2005strong} that the SDE $dX_t=b_1(X_t)dt+b_2(X_t)dt+dW_t,X_0=x$, with $b_1\in L_t^q([0,T];L_x^p(\mathbb{R}^d))$ with $\frac{d}{p}+\frac{2}{q}<1,$ and $b_2$ of linear growth, has a unique strong solution. Indeed, for weak well-posedness one just need to use Girsanov transform. In general, it is easier to deal with a drift in $L^\infty$ than a drift of $L_t^q-L_x^p$ type, as the former carries over to infinite dimensions neatly but the latter is based on some delicate Gaussian analysis which is clearly confined in finite-dimensions (if we set $d\to\infty$ we must have $p=\infty$ for $\frac{d}{p}+\frac{2}{q}<1$.)

We hope this technical difficulty can be resolved. This type of McKean-Vlasov SDEs can describe interactions that have both a local singularity and a long range correlation, which could be quite useful in physical modeling. 
\end{remark}

\section{A note on fractional Brownian motion}\label{chapter4}
Let $B^{H}=\left\{B_{t}^{H}, t \in[0, T]\right\}$ be a $d$-dimensional fractional Brownian motion with Hurst index $H \in(0,1)$. This means $B^H$ is a centered Gaussian process with $B^H_0=0$ and covariance given by
$$
R_{H}(t, s)=E\left(B_{t}^{H}\otimes B_{s}^{H}\right)=\frac{1}{2}\left\{|t|^{2 H}+|s|^{2 H}-|t-s|^{2 H}\right\}I_d .
$$

We say $B^H$ is singular if $H\in\left(0,\frac{1}{2}\right)$, and $B^H$ is regular if $H\in\left(\frac{1}{2},1\right)$. If $H=\frac{1}{2}$, $B^H$ is the Brownian motion. We will treat the singular and regular fractional cases separately in this paper.

The following results on Girsanov transform of fractional Brownian motion can be found in dimension 1 in \cite{NUALART2002103}, in multidimensional case in section 6 of \cite{le2020stochastic}, and in an infinite dimensional cylindrical case in \cite{bauer2019mckean}.

For $f \in L^{1}([a, b])$ and $\alpha>0$ the left fractional Riemann-Liouville integral of $f$ of order $\alpha$ on $(a, b)$ is given at almost all $x$ by
$$
I_{a^{+}}^{\alpha} f(x)=\frac{1}{\Gamma(\alpha)} \int_{a}^{x}(x-y)^{\alpha-1} f(y) \mathrm{d} y,
$$
where $\Gamma$ denotes the Euler function.

The fractional derivative can be introduced as inverse operation. We assume $0<\alpha$ $<1$ and $p>1$. We denote by $I_{a^{+}}^{\alpha}\left(L^{p}\right)$ the image of $L^{p}([a, b])$ by the operator $I_{a^{+}}^{\alpha}$. If $f \in I_{a^{+}}^{\alpha}\left(L^{p}\right)$, the function $\phi$ such that $f=I_{a^{+}}^{\alpha} \phi$ is unique in $L^{p}$ and it agrees with the left-sided Riemann-Liouville derivative of $f$ of order $\alpha$ defined by
$$
D_{a^{+}}^{\alpha} f(x)=\frac{1}{\Gamma(1-\alpha)} \frac{\mathrm{d}}{\mathrm{d} x} \int_{a}^{x} \frac{f(y)}{(x-y)^{\alpha}} \mathrm{d} y .
$$

\subsection{Girsanov Transform} We give a casual motivation for the operator $K_H^{-1}$ introduced in this section, following the presentation of \cite{NUALART2002103} and \cite{galeati2021noiseless}.

Given a Brownian motion $W$ on a probability space, we may construct a fractional Brownian motion $B^H$ on the same probability space that satisfies
$$B_t^H=\int_0^t K_H(t,s)dW_s,$$
where $K_H(t,s)$ is a square integrable kernel with precise expression given in \cite{NUALART2002103}.

Corresponding to this representation formula we define an operator $$K_H:L^{2}([0, T])\to I_{0^{+}}^{H+1 / 2}\left(L^{2}([0, T])\right)$$ such that $B^H=K_H(dW)$, where we identify $L^{2}([0, T])$ with the Cameron-Martin space. The precise expressions of $K_H$ are given for example in \cite{galeati2021noiseless} and \cite{NUALART2002103}. 

The operator $K_H$ can be inverted, which means for a fractional Brownian motion $B^H$, we may construct a Brownian motion $W$ on the same probability space such that $W_\cdot=\int_0^\cdot \left(K_H^{-1}B^H\right)_sds.$
The inverse operator
$$K_H^{-1}:I_{0^{+}}^{H+1 / 2}\left(L^{2}([0, T])\right)\to L^{2}([0, T])$$ has expressions (where $h'$ is the derivative of $h$)
\begin{equation}\label{dayu1.2}
K_{H}^{-1} h:=s^{H-\frac{1}{2}} D_{0+}^{H-\frac{1}{2}} s^{\frac{1}{2}-H} h^{\prime}\quad \text { if } H>1/2,
\end{equation}

\begin{equation}\label{xiaoyu}
K_{H}^{-1} h:=s^{1 / 2-H} D_{0^{+}}^{1 / 2-H} s^{H-1 / 2} D_{0^{+}}^{2 H} h \quad \text { if } H<1 / 2.
\end{equation}

If $h$ is absolutely continuous, for $H\in(0,\frac{1}{2})$ we can write the inverse operator as (see \cite{NUALART2002103}): 
\begin{equation}\label{1-1}
K_{H}^{-1} h=s^{H-\frac{1}{2}} I_{0^{+}}^{\frac{1}{2}-H} s^{\frac{1}{2}-H} h^{\prime}.
\end{equation}

From the inversion formula $W_\cdot=\int_0^\cdot \left(K_H^{-1}B^H\right)_sds$, we obtain
\begin{proposition}
[Theorem 2 of \cite{NUALART2002103}]
\label{thm1}Consider the shifted process $\tilde{B}_{t}^{H}=B_{t}^{H}+\int_{0}^{t} u_{s} \mathrm{~d} s$ defined by a process $u=\left\{u_{t}, t \in[0, T]\right\}$ with integrable trajectories. Assume that

(i) $\int_{0}^{*} u_{s} \mathrm{~d} s \in I_{0^{+}}^{H+1 / 2}\left(L^{2}([0, T])\right)$, almost surely.

(ii) $E\left(\xi_{T}\right)=1$, where
$$
\xi_{T}=\exp \left(-\int_{0}^{T}\left(K_{H}^{-1} \int_{0}^{\cdot} u_{r} \mathrm{~d} r\right)(s) \mathrm{d} W_{s}-\frac{1}{2} \int_{0}^{T}\left(K_{H}^{-1} \int_{0}^{\cdot} u_{r} \mathrm{~d} r\right)^{2}(s) \mathrm{d} s\right).
$$
Then the shifted process $\tilde{B}^{H}$ is an $\mathscr{F}_{t}^{B^{H}}$-fractional Brownian motion with Hurst index $H$ under the new probability $\tilde{P}$ defined by $\mathrm{d} \tilde{P} / \mathrm{d} P=\xi_{T}$.
\end{proposition}

\section{The singular fractional case }\label{chapter5}

\subsection{Proof of Theorem \ref{theorem05}}\label{sec5.3}
The proof of Theorem \ref{theorem05} is a slight modification of the proof of Theorem \ref{theorem01}. One major observation is that, by \eqref{1-1}, for $H\in(0,\frac{1}{2})$ and an adapted, absolutely continuous process $h$, we may find constants $C_H$ such that 
$|K_H^{-1}h|(t)\leq \sup_{0\leq s\leq t}|h'(s)|.$ Recall that in the Brownian case we have $K_{\frac{1}{2}}^{-1}h(t)=h'(t),$ so we can carry over many estimates from the $H=\frac{1}{2}$ case to the $H\in(0,\frac{1}{2})$ case.

For exponential moment estimates, we may use $\mathbb{E}[\|B^H\|^2_T]\leq C_H T^{2H}$, which follows from the self similarity property of fractional Brownian motion. $C_H$ is some finite constant depending only on $H$. Similarly, Fernique's theorem guarantees for some sufficiently small $c>0$, $\mathbb{E}e^{c\|B^H\|_T^2}<\infty$. Then estimates \eqref{moment}, \eqref{Repsilon}, \eqref{uniqueproof}, \eqref{uniformestimate}, \eqref{momentbar} and \eqref{7.11} carry over to the $H<1/2$ case, possibly with additional constants $C_H$ that depend on $H$.

In computing relative entropy between solutions of two SDEs, in the case of Brownian motion the formula is given in \eqref{Browniancase}. In the case of $B^H$, $H\in(0,\frac{1}{2})$, we combine \eqref{1-1} and Proposition \ref{thm1} to deduce that, for some constant $C_H<\infty$,
$$
\frac{d}{dt}H\left(P^{1}[t] \mid P^{2}[t]\right)\leq \frac{{C_H}}{2} \mathbb{E} \sup_{0\leq r\leq t}\left|b^{1}\left(r, Z^{1}\right)-b^{2}\left(r, Z^{1}\right)\right|^{2}.$$
This is no big difference, since we may still use the linear growth condition on $b$ and the data processing inequality \eqref{processing}.

Consequently, equations \eqref{7.6} and \eqref{7.21} carry over to the $H<\frac{1}{2}$ case with slight modifications. For example, equation \eqref{7.6} changes to
$$
H\left(\mu^{1}[t] \mid \mu^{2}[t]\right) \\
 \leq C_H \epsilon^{-1}\left(1+R_{\epsilon}(\mu^1,\mu^2)\right) \int_{0}^{t}\sup_{0\leq r\leq s} H\left(\mu^{1}[r] \mid \mu^{2}[r]\right) d s,
$$
and the same for \eqref{7.21}. The proof of Cauchy sequence in Theorem \ref{theorem01}, step 2 then carries over with the same minor change. The proof in step 3 is identical for $H=\frac{1}{2}$ and $H<\frac{1}{2}$. The other arguments are identical as those give in Theorem \ref{theorem01}.

\subsection{The case of distributional drift: proof of Proposition \ref{proposition1.50}}

The following two estimates will be necessary for the proof of Proposition \ref{proposition1.50}.

\begin{lemma}[Proposition 3.8 of \cite{galeati2021distribution}]\label{lemma5.1} Let $b\in L_T^q B_{\infty,\infty}^\alpha$ with $\alpha<0$ and $q>2$. Given that 
$\gamma:=1-\frac{1}{q}+\alpha H>\frac{1}{2}$, then for any $\tilde{\gamma}<\gamma$ we can find a function $K(\eta)<0$ such that
\begin{equation}
    \mathbb{E}\left[\exp\left(\frac{\eta}{\|b\|^2_{L_T^qB_{\infty,\infty}^\gamma}}\left\|\int_0^\cdot b(r,x+W_r)dr\right\|_{\tilde{\gamma}}^2\right)\right]\leq K(\eta)<\infty. 
\end{equation}
Here $\|\cdot \|_{\gamma}$ denotes the $\gamma$-Hölder norm of an $\mathbb{R}^d$-valued function on $[0,T]$.
\end{lemma}

In the following lemma, the operator $K_H$ has been defined in Section \ref{chapter4}, which is of critical importance when we do the Girsanov transform.

\begin{lemma}[Lemma 9 of \cite{galeati2021noiseless}]
\label{lemma5.2}
Consider an $\mathbb{R}^d$-valued process $h_t$ with $h_0=0.$ If $h\in\mathcal{C}_t^\beta$ for some $\beta>H+\frac{1}{2}$, then there is a constant $C=C(H,\beta,T)$ such that  $K_H^{-1}\in L_t^2$ and
$$\|K_H^{-1}h\|_{L^2}\leq C\|h\|_{\mathcal{C}^\beta}.$$
\end{lemma}

At this point we note that the assumption on $\alpha$ in Proposition \ref{proposition1.50} is precisely the one that ensures $1-\frac{1}{q}+\alpha H>H+\frac{1}{2}.$

We are now ready for the proof of Proposition \ref{proposition1.50}.

\begin{proof}

For any $\mu\in\mathcal{P}(\mathcal{C}([0,T];\mathbb{R}^d)$, denote by $\mu_t$ its marginal at time $t$, then the SDE
$$dX_t=B(t,X_t,\mu_t)dt+dB_t^H$$ has a unique weak solution with initial distribution $\mu_0$. We denote the solution by $X^\mu$. By a simple application of Girsanov transform, using Lemma \ref{lemma5.1} and conditioning on the initial law, one sees that 
\begin{equation}\label{qisiwole}
    \mathbb{E}\left[\exp\left(\frac{\eta}{\|b\|^2_{L_T^qB_{\infty,\infty}^\gamma}}\left\|\int_0^\cdot b(r,X^\mu_r)dr\right\|_{\tilde{\gamma}}^2\right)\right]\leq K(\eta)<\infty. 
\end{equation}
holds uniformly for any $b\in L_T^q B_{\infty,\infty}^\gamma$.

Now for any other $\nu\in\mathcal{P}(\mathcal{C}([0,T];\mathbb{R}^d)$, consider accordingly the solution $X^\nu$. As in the previous examples, we compute the relative entropy between the law of $X^\mu$ and $X^\nu$ as follows:
\begin{equation}\label{5.2345}\begin{aligned}
\|\operatorname{Law}(X^\mu)- \operatorname{Law}(X^\nu)\|^2_{TV}&\leq 2
    H(\operatorname{Law}(X^\mu)\mid \operatorname{Law}(X^\nu))\\&=\mathbb{E}\left[\left\|K_H^{-1}\left(\int_0^\cdot B(X_r^ \mu,\mu_r)-B(X_r ^ \mu,\nu_r)dr\right)\right\|^2\right]\\&\leq C  \left\|B(r,\cdot,\mu_r)-B(r,\cdot,\nu_r)\right\|^2_{L_T^q B_{\infty,\infty}^\alpha}\\&\leq C\left(\int _0^T h(r)^q \|\mu_r-\nu_r\|^q_{TV}dr\right)^\frac{2}{q}.
    \end{aligned}
\end{equation}
The function $h(t)$ is given in the statement of Proposition \ref{proposition1.50} and satisfies $h\in L^q([0,T]).$ In the third line of \eqref{5.2345} we used \eqref{qisiwole} and Lemma \ref{lemma5.2}, and in the fourth line we used our assumption on $B$. The proof then follows from standard Picard iteration.

\end{proof}

\section{The regular fractional case  }\label{chapter6}
This entire section is devoted to the proof of Theorem \ref{thm3}.
\subsection{Girsanov transform}\label{5.2} Assume that $H\in(\frac{1}{2},1)$.
For some adapted process $u_s$ satisfying $\int_{0}^{\cdot} u_{s} \mathrm{~d} s \in I_{0^{+}}^{H+1 / 2}\left(L^{2}([0, T])\right)$, we expand equation \eqref{dayu1.2} as follows:
\begin{equation}\label{fracexpansion}
\begin{aligned}
 &K_{H}^{-1}\left(\int_{0} u_{r} d r\right)(s) \\
&\quad = \frac{u_s s^{\frac{1}{2}-H}}{\Gamma\left(\frac{3}{2}-H\right)}+\frac{\left(H-\frac{1}{2}\right) s^{H-\frac{1}{2}}}{ \Gamma\left(\frac{3}{2}-H\right)}\int_{0}^{s} \frac{u_s s^{\frac{1}{2}-H}-u_r r^{\frac{1}{2}-H}}{(s-r)^{H+\frac{1}{2}}} d r \\
&\quad = \frac{u_s s^{\frac{1}{2}-H}}{\Gamma\left(\frac{3}{2}-H\right)}+\frac{\left(H-\frac{1}{2}\right) s^{H-\frac{1}{2}}}{ \Gamma\left(\frac{3}{2}-H\right)}\left(\int_{0}^{s}u_s \frac{s^{\frac{1}{2}-H}-r^{\frac{1}{2}-H}}{(s-r)^{H+\frac{1}{2}}} d r\right.\\
&\quad\quad\left.+\int_{0}^{s} r^{\frac{1}{2}-H} \frac{u_s-u_r}{(s-r)^{H+\frac{1}{2}}} d r\right) .
\end{aligned}
\end{equation}

Note that
$$\int_{0}^{1} \frac{u^{\frac{1}{2}-H}-1}{(1-u)^{H+\frac{1}{2}}} d u<\infty.$$

Since $H>\frac{1}{2}$, the terms involving $\int_0^s \frac{\left|u_s-u_r\right|}{(s-r)^{H+\frac{1}{2}}}dr$ cannot be bounded by the supremum norm of $u$. We should use instead the $\left(H-\frac{1}{2}+\epsilon\right)$- Hölder norm of $(u_s)_{0\leq s\leq t}$ as an upper bound, for some $\epsilon>0$. We will later fix an $\epsilon>0$ sufficiently small, and set
\begin{equation}\label{gammanew}\gamma(\epsilon):=H-\frac{1}{2}+\epsilon<1.\end{equation} 
For $\gamma\in(0,1)$ we denote by $\left\|u\right\|_{\gamma;[0,T]}$ the $\gamma$-Hölder norm of $u$ on the interval $[0,T]$, and likewise denote by $\|u\|_{\infty;[0,T]}$ the supremum norm of $u$ on $[0,T]$.
A direct computation shows for some constant $C_H$ depending on $T$ and $H$,
$$\int_0^s r^{\frac{1}{2}-H}\frac{1}{(s-r)^{1-\epsilon}}dr=C_H s^{\frac{1}{2}-H+\epsilon},\quad s\in[0,T].$$
Summarizing, we have  proved the following 
\begin{equation}\label{1277}\left|K_{H}^{-1}\left(\int_{0} u_{r} d r\right)(s)\right|\leq C_H \left(s^{\frac{1}{2}-H}\left\|u\right\|_{\infty;[0,s]}+s^\epsilon\|u\|_{\gamma;[0,s]}\right),\end{equation}

\subsection{Hölder conditions}
We will develop useful estimates under different assumptions on the drift. To simplify notations, let us label them by alphabets $\mathbf{A},\mathbf{B}$, etc.

$\left(\mathbf{A}\right)$with constants $1>\alpha>1-1/(2H)>0$ and $\beta>H-1/2>0$, there exists a finite constant $C>0$ such that 
$$
|b_0(t, x)-b_0(s, y)| \leqslant C\left(|x-y|^{\alpha}+|t-s|^{\beta}\right)\quad \text { for all } s, t \in[0, T], x, y \in \mathbb{R}^{d},
$$
$$
\begin{aligned}
&|b(t, x,x')-b(s, y,y')| \leqslant C\left(|x-y|^{\alpha}+|x'-y'|^{\alpha}+|t-s|^{\beta}\right) \\&\quad \text { for all }  s, t \in[0, T], x, y,x',y' \in \mathbb{R}^{d}.
\end{aligned}
$$

$\left(\mathbf{B^1}\right)$For some finite constant $C'>0$, $$|b_0(t,x)|\leq C',\quad |b(t,x,x')|\leq C'\text { for all } t \in[0, T], x, x' \in \mathbb{R}^{d}.$$

$\left(\mathbf{B^2}\right)$ There exists some $c>0$ such that $$\int e^{c|x|^2}\mu_0(dx)<M<\infty$$ for some $M>0$.

We make a choice of $\gamma$ as follows: choose $\epsilon>0$ sufficiently small in \eqref{gammanew} such that \begin{equation}\label{6.2.1}
H\alpha>\gamma \quad\text{ and }\quad  \beta>\gamma.\end{equation}

Now we study time regularity of the law of SDEs.
\begin{lemma}\label{lemmahelp} Assuming $(\mathbf{A})$ and $(\mathbf{B^1})$.
Consider the McKean-Vlasov SDE
$$
d X_{t}=\left(b_{0}\left(t, X_t\right)+\left\langle\mu_t, b\left(t, X_t, \cdot\right)\right\rangle\right) d t+d B_{t}^{H}.
$$
Then there exists a constant $C>0$ (depending only on $|b_0|_\infty$ and $|b|_\infty$) such that $$\mathcal{W}_1(\mu_s,\mu_t)\leq C|t-s|^H,\quad 0\leq s<t\leq T,$$ where $\mathcal{W}_1$ denotes the 1-Wasserstein distance on $\mathcal{P}(\mathbb{R}^d)$ with respect to the Euclidean distance.
 \end{lemma}
\begin{proof}
we have $|b_0|\leq C$ and $|b|\leq C$, this implies $$|X_t-X_s|\leq |B_t^H-B_s^H|+2C|t-s|,$$
then we have the estimate, with constant $C$ changing from line to line,
\begin{equation}\label{step1}
\mathcal{W}_1(\mu_s,\mu_t)\leq \mathbb{E}\left[\left|X_s-X_t\right|\right]\leq 2C\left|t-s\right|+C_H\left|t-s\right|^H\leq C_H\left(1+C\right)\left|t-s\right|^H,
\end{equation}using $\mathbb{E}[\left|B_s^H-B_t^H\right|^2]
\leq C_H|t-s|^{2H}$ and Cauchy-Schwartz.
\end{proof}

\begin{remark}
When $b_0$ and $b$ are not bounded but have linear growth, we still have \eqref{step1}. See \eqref{step10}.
\end{remark}

Since $b$ is $\alpha$-Hölder  in the last argument, for $\mu_1,\mu_2\in\mathcal{P}(\mathbb{R}^d)$, \begin{equation}
    \label{144}
\left|\left\langle \mu_1-\mu_2, b\left(t,x,\cdot\right)\right\rangle\right|\leq C\mathcal{W}_1^\alpha\left(\mu_1,\mu_2\right)\quad\text{ uniformly in $x\in\mathbb{R}^d$}, t\in[0,T].\end{equation}

For two paths $X,Y\in\mathcal{C}_T^d$, that $b$ satisfies assumption  $(\mathbf{A})$ implies:
\begin{equation}\label{triangletwo}\|b(\cdot,X_{\cdot},Y_{\cdot})\|_{\gamma;[0,T]}\leq_c \|X\|^\alpha_{\frac{\gamma}{\alpha};[0,T]}+\|Y\|^\alpha_{\frac{\gamma}{\alpha};[0,T]}.\end{equation}
The symbol $\leq_c$ means inequality holds with some constant depending on $H$, $\alpha$ and $\gamma$.

On the other hand, assume that $\mu\in\mathcal{P}(\mathcal{C}_T^d)$ satisfies  \begin{equation}
\label{specialrelation}\mathcal{W}_1\left(\mu_t,\mu_s\right)\leq C_H |t-s|^H\end{equation} for all $0\leq s\leq t\leq T$
then given $Y\in \mathcal{C}_T^d$, by triangle inequality we have
\begin{equation}\label{116}
\begin{aligned}
&\left|
    \left\langle \mu_t, b\left(t,Y_t,\cdot\right)\right\rangle-\left\langle \mu_s, b\left(s,Y_s,\cdot\right)\right\rangle\right|\\
&    \leq
    \left\langle \mu_t,\left|b\left(t,Y_t,\cdot\right)-b\left(s,Y_s,\cdot\right)\right|\right\rangle+
    \left|\left\langle \mu_t-\mu_s,b\left(s,Y_s,\cdot\right)\right\rangle\right|.
\end{aligned}    
\end{equation}
For the first term in the second line of inequality \eqref{116}, we use the Hölder continuity of $b$ in its time and space coordinates, as well as the relation between exponents \eqref{6.2.1}. For the second term we use the assumption \eqref{specialrelation} on $\mu$ as well as the estimate \eqref{144}.
 We can deduce that there exists a constant $K$ depending on $b_0$, $b$, $C_H$ such that for all $\lambda\in\mathbb{R}_+$, 
\begin{equation}
 \mathbb{E}\left[\exp\left(\lambda   \left\| t\mapsto \left\langle \mu_t, b\left(t,Y_t,\cdot\right)\right\rangle\right\|^2_{\gamma;[0,T]}\right)\right]\leq K\mathbb{E}\left[\exp\left(\lambda\|Y_\cdot\|^{2\alpha}_{\frac{\gamma}{\alpha};[0,T]}\right)\right].
\label{1234}
\end{equation}
The left hand side involves the Hölder norm of the function $t\mapsto \left\langle \mu_t, b\left(t,Y_t,\cdot\right)\right\rangle$.

In the following we  fix a $\delta$ such that \begin{equation}\label{6.4.1}H>\delta=\frac{\gamma}{\alpha}.\end{equation} The existence of such a $\delta$  follows from \eqref{6.2.1}.

Under the general, unbounded assumption $(\mathbf{B^2})$ we have the following result:

\begin{proposition}\label{prop14}

Assume that $(b_0,b,\mu_0)$ satisfy $\left(\mathbf{A}\right)$ and $\left(\mathbf{B^2}\right)$ and that the following McKean-Vlasov SDE admits a unique weak solution $X_t$ satisfying $\mathbb{E}[\|X\|_T]<M<\infty:$
$$
d X_{t}=\left(b_{0}(t, X_t)+\langle\mu_t, b(t, X_t, \cdot)\rangle\right) d t+d B_{t}^H, \quad \mu_t=\operatorname{Law}(X_t).
$$

Assume further that the law of $X_t$ satisfies \eqref{step1} for some constant $C>0$. Then we have an estimate
$$\mathbb{E}\left[\exp\left(\lambda \|X_\cdot\|^{2\alpha}_{\delta;[0,T]}\right)\right]\leq K(\lambda,M,C)<\infty,\quad \text{ for all }\lambda\in\mathbb{R}.$$
The constant $K(\lambda,M,C)<\infty$ depends on $\lambda,M,C$, $\mu_0$, $\alpha$ and $\beta$.
\end{proposition}
\begin{proof}

The proof is very simple if $b_0$ and $b$ are bounded by $C$. Indeed, in this case, \begin{equation}\label{boundedcase}\|X\|_{\delta;[0,T]}\leq \|B^H\|_{\delta;[0,T]}+2CT^{1-\delta}.\end{equation}
The conclusion then follows from Fernique's theorem since $\alpha<1$ and $\delta<H$. 

In the general case,
The idea is to use Girsanov transform. Denote by $\mathbb{P}$ the law of the McKean-Vlasov SDE and $\mathbb{Q}$ the law of $B_t^H+X_0$ on $\mathcal{C}_T^d$, where $X_0$ has law $\mu_0$ and independent of $B^H$. Once we have proved
\begin{equation}\label{equivmeasure}\mathbb{E}_\mathbb{Q}\left[\left(\frac{d\mathbb{Q}}{d\mathbb{P}}\right)^\kappa+\left(\frac{d\mathbb{P}}{d\mathbb{Q}}\right)^\kappa\right]<\infty,\quad \text{ for all }\kappa>0,\end{equation}
the claim follows from Cauchy-Schwartz and Fernique's theorem.

We write out the exponential martingale of $\frac{d\mathbb{Q}}{d\mathbb{P}}$, then use the estimate \eqref{1277} and the $\alpha$-Hölder continuity $(\alpha<1)$ of $b_0$ and $b$. It suffices to bound the $2\alpha$-order exponential moment of the supremum norm and $\delta$-Hölder norm of $B^H$, i.e., to prove that
$$\mathbb{E}[e^{\kappa \|B^H\|^{2\alpha}_{\infty;[0,T]}+\|B^H\|^{2\alpha}_{\delta;[0,T]}}]<\infty\quad \text{ for all }\kappa>0.$$
This is again a consequence of Fernique's theorem.
\end{proof}

\subsection{Proof of Theorem \ref{thm3}, in the case of bounded interactions}

First we prove weak well-posedness of the McKean-Vlasov SDE under the assumption that $b_0$ and $b$ are bounded. Then we generalize the results to the unbounded case.
 \begin{proposition}\label{>1/2}
Assume that $H\in(\frac{1}{2},1)$,  that $b_0$ and $b$ satisfy $(\mathbf{A})$ and $(\mathbf{B^1})$.
Then the Mckean-Vlasov SDE
$$
d X_{t}=\left(b_{0}(t, X_t)+\langle\mu_t, b(t, X_t, \cdot)\rangle\right) d t+d B_{t}^H, \quad \mu_t=\operatorname{Law}(X_t)
$$
admits a unique strong solution from any initial distribution $\mu_0$ having a finite second moment. \end{proposition}

\begin{proof}

Denote by $\mathcal{P}_{\text{Hol}}(C)(\mathcal{C}_T^d)$ the subset of $\mathcal{P}(\mathcal{C}_T^d)$ consisting of measures $\mu$ satisfying

\begin{equation} \label{step2}
    \mathcal{W}_1(\mu_s,\mu_t)\leq C |t-s|^H,\quad 0\leq s<t\leq T.
\end{equation}

Step 1. For any probability measure $\mu\in\mathcal{P}(\mathcal{C}_T^d)$ satisfying the time regularity estimate \eqref{step2}, 
define $\bar{b}_{\mu}(t, x):=b_{0}(t, x)+\langle\mu_t, b(t, x, \cdot)\rangle$, then the coefficient $\bar{b}_\mu(t,x)$ is $\alpha$-Hölder in space and $\beta$-Hölder in time, with $\alpha$ and $\beta$ satisfying the assumption in $(\mathbf{A})$. Thus the SDE $dX_t=\bar{b}_\mu(t,X_t)dt+dB_t^H$ has a unique strong solution starting from the initial law $\mu_0$. This follows from Theorem 3 and 4 in \cite{NUALART2002103}. 

Moreover, under assumption $(\mathbf{B^1})$, by Lemma \eqref{lemmahelp} the law of $X_t$ satisfies the same regularity estimate \eqref{step2}, with the constant $C$ depending only on $|b_0|_\infty$ and $|b|_\infty$. We fix once and for all this value of $C$ that will show up in \eqref{step2} throughout the end of the proof.

Step 2. For $\mu \in \mathcal{P}\left(\mathcal{C}_{T}^{d}\right)$ satisfying \eqref{step2}, let $\Phi(\mu) \in \mathcal{P}\left(\mathcal{C}_{T}^{d}\right)$ denote the law of the unique solution of the $\operatorname{SDE}\left(\bar{b}_{\mu}\right)$, started from $\mu_{0}$.

We denote by $\mu^0$ the law of $B^H$ on $\mathcal{C}_T$, define $\mu^1:=\Phi(\mu^0)$, and inductively define $\mu^i:=\Phi(\mu^{i-1})$ for all $i\in\mathbb{N}$. Then step 1 implies $\mu^i$ satisfies \eqref{step2} for each $i$. We can now apply Girsanov transform and obtain: (Novikov's condition will be checked in the proof)
\begin{equation}
\begin{aligned}    
&H\left(\mu^{i+2}[t]\mid \mu^{i+1}[t]\right)=H\left(\Phi(\mu^{i+1})[t] \mid \Phi(\mu^i)[t]\right)\\
&\quad\leq C\int_{\mathcal{C}_{t}^{d}} \int_0^t \left|K_H^{-1}\left(\int_0^\cdot \bar{b}_{\mu^{i+1}}(r, x)-\bar{b}_{\mu^i}(r, x)dr\right)(s)\right|^2 ds \mu^{i+2}[t](dx).
\end{aligned}\label{expansiongir}
\end{equation}

To bound $K_H^{-1}$, we need to control both $\|u\|_\infty$ and $\|u\|_{\gamma;[0,T]}$, see \eqref{1277}.

Fix $0\leq s\leq T$ and define a functional \footnote{We briefly discuss the motivation for such a functional $F$. Checking the computation in \eqref{fracexpansion}, one sees that the second term in the bracket that appears after the second equality in \eqref{fracexpansion} is not controlled by the supremum norm $\|u\|_{\infty,[0,T]}$, but the other two are controlled by the supremum norm, which is easier. Therefore, we use the functional $F$ to single out this term and pay particular attention to it.} $F:\mathcal{C}_s^d\times \mathcal{C}_s^d\to\mathbb{R}^d$:
\begin{equation}\label{pqrs} F(X,Y):=\int_{0}^{s} r^{\frac{1}{2}-H} \frac{\left|b(s,X_s,Y_s)-b(r,X_r,Y_r)\right|}{(s-r)^{H+\frac{1}{2}}} d r.\end{equation}

$F$ may not be well defined if $X$ and $Y$ do not have enough Hölder regularity, but this will not affect the proof, which will be clear in a moment.

For $F(X,Y)$ we have the following crude upper bound
\begin{equation}\label{hier2}
\begin{aligned}
&\left|F(X,Y)\right|\leq \|b(\cdot,X_\cdot,Y_\cdot)\|_{H-\frac{1}{2}+\epsilon;[0,T]}\int_0^s \frac{r^{\frac{1}{2}-H}}{(s-r)^{1-\epsilon}}dr\\
&\quad\leq C_H s^{\frac{1}{2}-H+\epsilon}\|b(\cdot,X_\cdot,Y_\cdot)\|_{H-\frac{1}{2}+\epsilon;[0,T]}.
\end{aligned}
\end{equation}

Assume that for each $i$, $X^i$ has law $\mu^i$. By the bound \eqref{triangletwo}, equation \eqref{boundedcase} and Fernique's theorem, there exists a constant $K_H(\lambda)$ depending on $\lambda$, $b_0$ and $b$ such that for each pair $(i,j)\in\mathbb{N}_+\times\mathbb{N}_+$, 
\begin{equation}
    \label{uniqueness}
\mathbb{E}\left[\exp\left(\lambda   \left\|   b\left(\cdot,X_\cdot^i,X_\cdot^j\right)\right\|^2_{\gamma;[0,T]}\right)\right]\leq K_H(\lambda)<\infty \text{ for each } \lambda\in\mathbb{R}_+.\end{equation}

Now we apply the weighted Pinsker inequality \eqref{pinsker} and deduce that (Note: a careful check of its proof shows that the $f$ in \eqref{pinsker} only needs to be well-defined on the support of $\nu$ and $\nu'$, and we assume $\nu$ absolutely continuous w.r.t. $\nu'$ throughout the argument.)
$$
\begin{aligned}
&\left\langle \mu^{i+1}[s]-\mu^i[s],F(X^{i+2},\cdot)\right\rangle^2\\&\quad\leq 2\left(1+\log[\int_{\mathcal{C}_s^d} e^{|F(X^{i+2},y)|^2}\mu^i(dy)]\right)H\left(\mu^{i+1}[s]\mid\mu^i[s]\right).
\end{aligned}
$$

Taking expectation with respect to $\mu^{i+2}(dx)$, 
$$
\begin{aligned}
&\int_{\mathcal{C}_s^d} \left\langle \mu^{i+1}[s]-\mu^i[s],F(x,\cdot)\right\rangle^2\mu^{i+2}(dx)\\&\quad
\leq\int_{\mathcal{C}_s^d}  2\left(1+\log\left[\int_{\mathcal{C}_s^d} e^{|F(x,y)|^2}\mu^i(dy)\right]\right)\mu^{i+2}(dx)H\left(\mu^{i+1}[s]\mid \mu^i[s]\right).
\end{aligned}
$$
The exponential moment control \eqref{uniqueness} and estimate \eqref{hier2} imply that, for some constant $C_H$ depending on $T$ and $H$,
\begin{equation}\label{tyui}
\int_{\mathcal{C}_s^d} \left\langle \mu^{i+1}[s]-\mu^i[s],F(x,\cdot)\right\rangle^2\mu^{i+2}(dx)\\
\leq C_H s^{1-2H+2\epsilon}  H\left(\mu^{i+1}[s]\mid\mu^i[s]\right).
\end{equation}
Now apply \eqref{expansiongir} and the computation at the beginning of Section \ref{5.2}, we conclude that for some $C_H>0$,
$$
H\left(\mu^{i+2}[t]\mid \mu^{i+1}[t]\right)
\leq C_H\int_0^t s^{1-2H} H\left(\mu^{i+1}[s]\mid \mu^{i}[s]\right)ds\quad\text{ for all }0\leq t\leq T.
$$

Step 3. 
Since $H\in(\frac{1}{2},1)$, $ t^{1-2H}\in L^1([0,T])$. Pinsker's inequality implies there exists a measure $\mu\in\mathcal{P}(\mathcal{C}_T^d)$ such that $\mu^n$ converges to $\mu$ in total variation.

Since $\mu_0$ has a finite second moment and $b_0$, $b$ are bounded, we check that $\mu^n$ has finite second moment on $\mathcal{C}_T^d$ that is bounded uniformly in $n$. Following the proof in Theorem \ref{theorem01}, step 3, we can upgrade the convergence of $\mu^n$ to $\mu$ to allow test functions of linear growth. For each $0\leq s\leq t\leq T$, the function $|X_s-X_t|$ on $\mathcal{C}_T^d$ has linear growth, therefore each $\mu^n$ satisfying $\mathbb{E}\left|X_s^n-X_t^n\right|\leq C |t-s|^H$ for $0\leq s\leq t\leq T$
implies that $\mu$ satisfies \eqref{step2} with the same constant $C$ fixed at the end of Step 1.

Step 4. 
For this fixed $\mu\in\mathcal{P}(\mathcal{C}_T^d)$,
$b_{0}(t, x)+\langle\mu_t, b(t, x, \cdot)\rangle$ is Hölder continuous of order $1>\alpha>1-2H$ in $x$ and of order $\gamma>H-\frac{1}{2}$ in $t$. Again by Theorem 3 and 4 of \cite{NUALART2002103}, the SDE $d X_{t}=\left(b_{0}(t, X_t)+\langle\mu_t, b(t, X_t, \cdot)\rangle\right) d t+d B_{t}^H$ has a unique weak solution with initial distribution $\mu_0$. Denote by $\widetilde{\mu}\in\mathcal{P}(\mathcal{C}_T^d)$ the law of this weak solution. We claim that $\widetilde{\mu}=\mu$.

Since $b$ is bounded measurable, for any $t$ and $x$, $\left\langle \mu_t^n, b(t,x,\cdot)\right\rangle$ converges to $\left\langle \mu_t, b(t,x,\cdot)\right\rangle$. Then by dominated convergence theorem, $\left|\int_{\mathcal{C}_t^d}\left\langle \mu_t^n-\mu_t, b(t,x,\cdot)\right\rangle\right|^2\widetilde{\mu}(dx)$ converges to 0.
At this point we need a different estimate of \eqref{1277} which is not difficult to obtain: for some $\theta>0$ sufficiently small, for any adapted process $u_s$ satisfying $\int_{0}^{\cdot} u_{s} \mathrm{~d} s \in I_{0^{+}}^{H+1 / 2}\left(L^{2}([0, T])\right)$, 
\begin{equation}
\left|K_{H}^{-1}\left(\int_{0}^\cdot  u_{r} d r\right)(s)\right| \leq C_H s^{\frac{1}{2}-H}\left(\|u\|_{\infty;[0,T]}+\|u\|^\theta_{\infty;[0,T]}\|u\|^{1-\theta}_{\gamma;[0,T]}\right).\label{thetabound}
\end{equation}

Defining the random variable $u^n_r:=\left\langle \mu_r^n-\mu_r, b(r,X_r,\cdot)\right\rangle$, then we have shown that \begin{equation}\int_{\mathcal{C}_T^d}\|u^n_\cdot\|_{\infty;[0,T]}^2\widetilde{\mu}(dX)\to 0\label{normalbound},\quad n\to\infty.\end{equation} 
Although not stated explicitly, $u^n$ depends on the sample path $(X_r)_{0\leq r\leq T}.$

We have as in equation \eqref{1234} a uniform in $n$ bound 
\begin{equation}\mathbb{E}_{\widetilde{\mu}}\left[\exp\left(\lambda\|u^n_\cdot\|^2_{\gamma;[0,T]}\right)\right]<C(\lambda)<\infty,\quad \lambda\in\mathbb{R}.\label{expbound}\end{equation}

Now we compute relative entropy between $\widetilde{\mu}$ and $\mu^{n+1}$ via Girsanov transform:
$$H\left(\widetilde{\mu}[t]\mid \mu^{n+1}[t]\right)=\frac{1}{2}\int_0^t \left|K_H^{-1}\left(\int_0^\cdot u^n_rdr\right)(s)\right|^2 ds\widetilde{\mu}(dx).$$
 
Using Cauchy-Schwartz and combine equations \eqref{expbound}, \eqref{normalbound} and \eqref{thetabound}, we obtain $$H\left(\widetilde{\mu}[t]\mid \mu^{n+1}[t]\right)\to 0, \quad n\to\infty.$$

Pinsker's inequality implies $\mu^{n}$ converges to $\widetilde{\mu}$ in total variation, but $\mu^n$ also converges to $\mu$ in total variation, so $\widetilde{\mu}=\mu$ and $\mu$ is a solution of the McKean-Vlasov SDE starting from $\mu_0$.

Step 5. To justify uniqueness, for $\mu^1,\mu^2\in\mathcal{P}(\mathcal{C}_T^d)$ two weak solutions with initial distribution $\mu_0$,  the various estimates obtained in the existence proof imply that
$$H\left(\mu^1[t]\mid \mu^2[t]\right)\leq C\int_0^t s^{1-2H}H\left(\mu^1[s]\mid\mu^2[s]\right)ds.$$
Since $\mu^1[0]=\mu^2[0]=\mu_0$, a standard Gronwall argument implies $\mu_1=\mu_2$ on $[0,T]$.
\end{proof}

\subsection{Proof of Theorem \ref{thm3}, in the case of unbounded interactions}

Now we are in the position to prove Theorem \ref{thm3}.
\begin{proof}
Step 1. We truncate the interaction and establish exponential moment estimates. Define $b^{n}:=(b \wedge n) \vee(-n)$ for each $n \in \mathbb{N}$. ( $b$ is a $d$-dimensional vector, we take min and max in each coordinate.) By Proposition \ref{>1/2}, boundedness of $b^{n}$, and the fact that Hölder continuity is preserved after truncation, there exists a unique weak solution $X^{n} \sim \mu^{n}$ of the McKean-Vlasov equation with initial law $X_0^n\sim \mu_0,$
$$
d X_{t}^{n}=\left(b_{0}\left(t, X_t^{n}\right)+\left\langle\mu_t^{n}, b^{n}\left(t, X_t^{n}, \cdot\right)\right\rangle\right) d t+d B_{t}^H,  \quad \mu_t^{n}=\operatorname{Law}\left(X_t^{n}\right).$$
We have as in \eqref{moment} that there exist $c, C_{1}>0$ such that
\begin{equation}\label{momentnn}
\sup _{n} \int_{\mathcal{C}_{T}^{d}} e^{c\|x\|_{T}^{2}} \mu^{n}(d x) \leq C_{1}.
\end{equation}
From this uniform moment estimate we deduce that, if we choose $\epsilon$ sufficiently small, (recall the definition of $R_\epsilon(\mu,\nu)$ in \eqref{Repsilon}),
\begin{equation}
R_\epsilon:=\sup_{n,m} R_\epsilon(\mu^n,\mu^m)<\infty.\end{equation}

Combining 
\begin{itemize} \item the elementary inequality
$$
\mathbb{E}\left[\left|X_t^n-X_s^n\right|\right]\leq \mathbb{E}\left|B_t^H-B_s^H\right|+|t-s|\sup_{s\leq r\leq t}\mathbb{E}\left|\left\langle\mu_r^n,b_0+b^n(r,X_r^n,\cdot)\right\rangle\right|],$$

 \item  the linear growth property of $b_0$ and $b$, and \item  the uniform moment control \eqref{momentnn},\end{itemize} we deduce that there exists a constant $C_H>0$ and a constant $C>0$ depending on the drift $b_0$, interaction $b$ and constants $c$, $C_1$ such that  \begin{equation}\label{step10}
\mathcal{W}_1(\mu^n_s,\mu^n_t)\leq  2C\left|t-s\right|+C_H\left|t-s\right|^H\leq_c\left|t-s\right|^H\quad\text{ for all } n\in\mathbb{N}_+, 0\leq s\leq t\leq T.\end{equation}

Now we apply Proposition \ref{prop14} and obtain a refined estimate: for some constant $K(\lambda,C_1)$,
$$\mathbb{E}[\exp(\lambda \|X^n_\cdot\|^{2\alpha}_{\delta;[0,T]})]\leq K(\lambda,C_1)<\infty,\quad \text{ for all } \lambda\in\mathbb{R}_+,n\in\mathbb{N}_+.$$

Following the same argument as in \eqref{116} and \eqref{1234}, the Hölder continuity of $b$ implies that uniform in the choice of integers $n,m\geq 1$, for some constant $K_H(\lambda)>0$,
\begin{equation}\label{unitm}
 \mathbb{E}\left[\exp\left(\lambda   \left\| t\mapsto \left\langle \mu^n_t, b\left(t,X_t^m,\cdot\right)\right\rangle\right\|^2_{\gamma;[0,T]}\right)\right]\leq C_HK_H(\lambda)<\infty \text{ for all } \lambda\in\mathbb{R}.
\end{equation}

Step 2. We show $(\mu^n)_{n\in\mathbb{N}}$ is a Cauchy sequence in total variation distance.

For $n,m\in\mathbb{N}_+$ define $\bar{b}^m_{\mu^n}(t, x):=\langle\mu_t^n, b^m(t, x_t, \cdot)\rangle$, and define $\bar{b}_{\mu^n}(t, x):=\langle\mu_t^n, b(t, x_t, \cdot)\rangle$, we bound relative entropy as follows:
\begin{equation}
\begin{aligned}    
&H\left(\mu^{n}[t]\mid \mu^{m}[t]\right)\\
&\quad\leq \frac{1}{2}\int_{\mathcal{C}_{t}^{d}} \int_0^t \left|K_H^{-1}\left(\int_0^\cdot \bar{b}^n_{\mu^{n}}(r, x_r)-\bar{b}^m_{\mu^m}(r, x_r)dr\right)(s)\right|^2 ds \mu^{n}[t](dx).
\end{aligned}
\end{equation}

Recall \eqref{1277} that to bound $K_H^{-1}$ we need to control both the supremum norm and the $\gamma$- Hölder norm of the terms involved.
For the Hölder norm we argue as in equations \eqref{pqrs},  \eqref{uniqueness} and  \eqref{tyui}, for the supremum norm we argue as in \eqref{7.6}, we obtain that, for some $C_H>0$,

$$\begin{aligned}&\int_{\mathcal{C}_{t}^{d}} \int_0^t \left|K_H^{-1}\left(\int_0^\cdot \bar{b}_{\mu^{n}}(r, x_r)-\bar{b}_{\mu^m}(r, x_r)dr\right)(s)\right|^2 ds \mu^{n}[t](dx)\\
&\quad\leq C_H\int_0^t s^{1-2H}H\left(\mu^n[s]\mid\mu^m[s]\right)ds
.\end{aligned}$$

The last step is to bound  
$$\begin{aligned}&\int_{\mathcal{C}_{t}^{d}} \int_0^t \left|K_H^{-1}\left(\int_0^\cdot \bar{b}_{\mu^{n}}(r, x_r)-\bar{b}^m_{\mu^n}(r, x_r)dr\right)(s)\right|^2 ds \mu^{n}[t](dx)
.\end{aligned}$$
It suffices to control 
$$\begin{aligned}&\int_{\mathcal{C}_{t}^{d}} \int_0^t \left|K_H^{-1}\left(\int_0^\cdot \bar{b}_{\mu^{n}}(r, x_r)dr\right)(s)\right|^2 ds 1_{\sup_{0\leq s\leq t}\left|\bar{b}_{\mu^n}(s,x_s)\right|\geq m} \mu^{n}[t](dx)
.\end{aligned}$$
This term vanlshes as $m\to\infty$ uniform in $n\in\mathbb{N}_+$ (a consequence of \eqref{momentn} and \eqref{unitm}.) Then we conclude from triangle inequality
$$\left|\bar{b}_{\mu^n}^n-\bar{b}_{\mu^m}^m\right|\leq \left|\bar{b}_{\mu^n}^n-\bar{b}_{\mu^n}\right|+\left|\bar{b}_{\mu^m}^m-\bar{b}_{\mu^m}\right|+\left|\bar{b}_{\mu^n}-\bar{b}_{\mu^m}\right|$$
and Gronwall's lemma that $$ H\left(\mu^{n}[t]\mid \mu^{m}[t]\right)\to 0,\quad n,m\to\infty.$$

Step 3. Pinsker's inequality implies that there exists a $\mu\in\mathcal{P}(\mathcal{C}_T^d)$ such that $\mu^n$ converges to $\mu$ in total variation. 
If $f$ is a function on $\mathcal{C}_T$ of linear growth, i.e. $\sup_x |f(x)|/(1+\|x\|_T)<\infty$, then $\langle f,\mu^n\rangle\to\langle f,\mu\rangle.$ The proof is identical to the Brownian case (proof to Theorem \ref{theorem01}, step 3), making use of estimate \eqref{momentnn}.

 The same reasoning as in Step 3 of proof to Proposition \ref{>1/2} implies $\mu$ also satisfies \eqref{step10}.

Step 4. For this fixed $\mu$,
$b_{0}(t, x)+\langle\mu_t, b(t, x, \cdot)\rangle$ is Hölder continuous of order $1>\alpha>1-2H$ in $x$ and of order $\gamma>H-\frac{1}{2}$ in $t$. Then the SDE $d X_{t}=\left(b_{0}(t, X_t)+\langle\mu_t, b(t, X_t, \cdot)\rangle\right) d t+d B_{t}^H$ has a unique weak solution starting from $\mu_0$. Denote by $\widetilde{\mu}$ the law of this weak solution. We claim that $\widetilde{\mu}=\mu$.
 
 For each fixed $(t,x)$, $b(t,x,\cdot)$ has linear growth in the third argument, so $\langle \mu_t^n-\mu_t, b(t,x,\cdot)\rangle\to 0$ pointwise. Meanwhile, $|\langle \mu_t^n-\mu_t, b(t,x,\cdot)\rangle|\leq C(1+\mathbb{E}_{\mu^n}[\|X\|_T]+\mathbb{E}_{\mu}[\|X\|_T])<\infty$, so by bounded convergence theorem, 
 $$\left|\left\langle \mu_t^n-\mu_t, b(t,x,\cdot)\right\rangle\right|^2\widetilde{\mu}(dx)\to 0\quad\text{ uniformly on }[0,T].$$
 
 Setting $u^n_r:=\left\langle \mu_r^n-\mu_r, b(r,X_r,\cdot)\right\rangle$, we have the same estimates as in \eqref{thetabound}, \eqref{normalbound} and \eqref{expbound}. Consequently we obtain
$$\int_0^t \left|K_H^{-1}(\int_0^\cdot u_r^ndr)(s)\right|^2 ds\widetilde{\mu}(dx)\to 0, \quad n\to\infty.$$

Setting $v^n_r:=\left\langle \mu_r^n, b^n-b(r,X_r)\right\rangle$, then uniformly in $r\in[0,t]$, $\left\langle |v_r|^2,\widetilde{\mu}\right\rangle \to 0$ as $n\to\infty$ since $b$ has linear growth and since we have the uniform moment estimate \eqref{momentnn}. Moreover, $v^n_r$ also satisfies \eqref{expbound}, from which we deduce that
$$\int_0^t \left|K_H^{-1}(\int_0^\cdot v^n_rdr)(s)\right|^2 ds\widetilde{\mu}(dx)\to 0, \quad n\to\infty.$$

Now we compute relative entropy via Girsanov transform. The aforementioned two estimates imply that
$$H\left(\widetilde{\mu}[t]\mid \mu^{n}[t]\right)=\frac{1}{2}\int_0^t \left|K_H^{-1}(\int_0^\cdot 
(u^n_r+v^n_r)dr)(s)\right|^2 ds\widetilde{\mu}(dx)\to 0, \quad n\to\infty.$$ Pinsker's inequality implies $\mu^{n}$ converges to $\widetilde{\mu}$ in total variation, but $\mu^n$ also converges to $\mu$ in total variation, so $\widetilde{\mu}=\mu$ and $\mu$ is a solution of the McKean-Vlasov SDE starting from $\mu_0$.

Step 5. To prove uniqueness, let $X^i\sim \mu^i$, $i=1,2$ two solutions of the McKean-Vlasov SDE with initial law $\mu_0$. From the condition on $\mu_0$ and the linear growth property of $b$, we may find $c>0$ small enough and $C<\infty$ that
\begin{equation}\label{1701}
\int_{\mathcal{C}_{T}^{d}} e^{c\|x\|_{T}^{2}} \mu^{i}(d x)=\mathbb{E} e^{c\left\|X^{i}\right\|_{T}^{2}} \leq C,\quad i=1,2.
\end{equation}

  Consequently, \eqref{step10} and \eqref{unitm} also hold for $\mu^i$ and $X^i$, $i=1,2$ in place of $\mu^n$ and $X^m$.

We compute the relative entropy between $\mu^1$ and $\mu^2$:
$$
H\left(\mu^{1}[t] \mid \mu^{2}[t]\right) =\frac{1}{2} \int_{\mathcal{C}_{T}^{d}} \int_{0}^{t}\left|K_H^{-1}\left(\int_0^\cdot\left\langle\mu^{1}-\mu^{2}, b(r, x, \cdot)\right\rangle dr\right)(s)\right|^{2} \mu^1(dx)d s.$$
Now we bound the term involving $K_H^{-1}$, recalling equation \eqref{1277}. For the Hölder norm we invoke the same computation as in \eqref{tyui}, and for the supremum norm invoke the same computation as in  \eqref{7.6}. Then $\mu^1[t]=\mu^2[t]$ follows from a standard application of Gronwall's lemma and $\mu^1[0]=\mu^2[0]$.

This concludes the proof of Theorem \ref{thm3}.\end{proof}

\section{Mean Field SPDEs}\label{chapter 7}

We first recall the basic notions of Girsanov transform for Wiener processes defined on general Hilbert space $H$. Given $W$ a cylindrical Wiener process on a Hilbert space $H$, we may find $\{e_k\}$ a complete orthonormal basis in $H$ and $\{\beta_k\}$ a series of independent real-valued Wiener processes, such that
$$\langle W(t),x\rangle=\sum_{k=1}^\infty \beta_k(t)\langle x,e_k\rangle,\quad x\in H.$$

Denote by $\mathcal{F}_t$ the filtration generated by $W(t)$, we have the following version of Girsanov transform (see for example Theorem 10.14 of \cite{da2014stochastic}):

\begin{proposition}
Given a $H$-valued, $\mathcal{F}_t$-predictable process $\psi(\cdot)$ satisfying 
\begin{equation}\label{novikovcondition}
    \mathbb{E}\left(e^{\int_0^T \langle \psi(s),dW_s\rangle_H-\frac{1}{2}\int_0^T \|\psi(s)\|_H^2ds}\right)=1,
\end{equation}
then the process
\begin{equation}
    \widehat{W}(t)=W(t)-\int_0^t \psi(s)ds,\quad t\in[0,T]
\end{equation}
 is a cylindrical Wiener process on $H$ with respect to $(\mathcal{F}_t)_{t\geq 0}$ on the probability space $(\Omega,\mathcal{F},\widehat{\mathbb{P}})$, where
 \begin{equation}
     d\widehat{\mathbb{P}}(\omega)=e^{\int_0^T \langle \psi(s),dW(s)\rangle_H-\frac{1}{2}\int_0^T \|\psi(s)\|_H^2 ds}d\mathbb{P}(\omega).
 \end{equation}
\end{proposition}

The assumption \eqref{novikovcondition} holds if the following Novikov type condition is justified (see for example Proposition 10.17 of \cite{da2014stochastic}): for some $\delta>0$, $$\sup_{t\in[0,T]}\mathbb{E}\left(e^{\delta\|\psi(t)\|^2_H}\right)<+\infty.$$

We are now ready to compare the laws of two SPDEs with the same diffusion coefficient. Consider two equations
$$
    dX=AX dt+B(X)dW(t),\quad X(0)=x,
$$
$$d\widetilde{X}=(A\widetilde{X}+\widetilde{F}(\widetilde{X}))dt+B(\widetilde{X})dW(t),\quad \widetilde{X}(0)=x,$$
then following the proof of Theorem 10.18 in \cite{da2014stochastic}, i.e. writing the solution $X$ and $\widetilde{X}$ in mild formulation and comparing Girsanov density, one obtains: if for some $\delta>0$,
\begin{equation} \label{7.414}
    \sup_{t\in[0,T]}\mathbb{E}\left(e^{\delta \|G^{-1}X(t)[\widetilde{F}(X_t)]\|_H^2)}\right)<+\infty,
\end{equation}
then denote by $\psi(t)=B^{-1}(X(t))\widetilde{F}(X_t)$, we have hat for any Borel set $\Gamma\in \mathcal{C}([0,T];H)$, 
\begin{equation}\label{7.5555}
    \mathbb{P}(\widetilde{X}(\cdot)\in\Gamma)=\mathbb{E}\left(e^{\int_0^T \langle \psi(s),dW(s)\rangle_H-\frac{1}{2}\int_0^T \|\psi(s)\|_H^2ds};X(\cdot)\in\Gamma.
    \right)
\end{equation}

\subsection{Solving McKean-Vlasov SPDEs}

We are now ready to solve McKean-Vlasov type SPDEs, and we start from the stochastic heat equation.

\subsubsection{Stochastic heat equation}\label{she123}

Recall that we work on the Hilbert space $H:=L^2([0,1];\mathbb{R})$ and we consider the nonlinearity $G:[0,T]\times H\times\mathcal{P}(H)\to H$ satisfying, for some $M>0,$
$$\|G(t,x,\mu_t)\|_H\leq M,\quad \text{ for any }x\in H,\mu_t\in\mathcal{P}(H),t\in[0,T],$$
$$\|G(t,x,\mu_t)-G(t,x,\nu_t)\|_H\leq M\|\mu_t-\nu_t\|_{TV},\quad\text{ for any } \mu_t,\nu_t\in\mathcal{P}(H),$$
uniformly over $t\in[0,T]$ and $x\in H$. 

Now consider two probability laws $\mu,\nu\in\mathcal{P}(\mathcal{C}([0,T];H))$, whose marginals at time $t$ are denoted $\mu_t$, $\nu_t$ respectively. Let $X^\mu$ solve the SPDE
$$\frac{\partial}{\partial t}X^\mu(t)=\frac{\partial^2}{\partial\sigma^2}  X^\mu(t) dt+G(t,X^\mu(t),\mu_t)dt+dW(t)$$
and $X^\nu$ solves the SPDE
$$\frac{\partial}{\partial t}X^\nu(t)=\frac{\partial^2}{\partial\sigma^2}  X^\nu(t) dt+G(t,X^\nu(t),\nu_t)dt+dW(t),$$
with the same initial law $\mu_0\in\mathcal{P}(H)$. Probabilistic weak well-posedness of these two SPDEs follow by Girsanov transform since $G$ is bounded, and probabilistic strong well-posedness also holds, see \cite{gyongy1993quasi}, though we will not use this fact in this paper.

Since $G$ is bounded, condition \eqref{7.414} is automatically justified and we can proceed as follows:
\begin{equation}\begin{aligned}
    \|\operatorname{Law}(X^\mu)-\operatorname{Law}(X^\nu)\|_{TV}^2&\leq 2 H\left(\operatorname{Law}(X^\mu)\mid\operatorname{Law}(X^\nu)
    \right)\\&=\int_0^T \mathbb{E}\left[\|G(s,X^\mu(s),\mu_t)-G(s,X^\mu(s),\nu_s)\|_H^2\right]ds\\&\leq M\int_0^T  \|\mu_s-\nu_s\|^2_{TV}ds.
    \end{aligned}
\end{equation}
Again the first line follows from Pinsker's inequality, the second line from Girsanov transform \eqref{7.5555} and the last line from our assumption on $B$. We now finish the construction of a solution to the McKean-Vlasov SPDE 
$$\frac{\partial}{\partial t}X(t)=\frac{\partial^2}{\partial\sigma^2} X(t)dt+G(t,X(t),\mu_t)dt+dW(t),\quad \operatorname{Law}(X(t))=\mu_t,$$
from a standard Picard iteration scheme. This proves Theorem \ref{theorem1.71.7}, (1).

\subsubsection{White noise acting on the boundary}

The following example is inspired by Example 2.5 of \cite{maslowski2000probabilistic}, see also \cite{da1993evolution} and Chapter 13 of \cite{da1996ergodicity}. In this case denote by $H:=L^2([0,1];\mathbb{R}^2)$.

Consider the one-dimensional heat equation with nonlinear boundary condition
\begin{equation}\label{wihteboundary}
    \begin{cases}
    \frac{\partial u}{\partial t} =\frac{\partial^2 u}{\partial x^2},\quad t>0,x\in [0,1],\\
    (\frac{\partial u}{\partial x}(t,0)=\frac{\partial u}{\partial x}(t,1))=f(u(t))+\dot{\omega}(t),\quad t\geq 0,
    \end{cases}
\end{equation}
  where $\omega$ is  two-dimensional Brownian motion and $f:H\to \mathbb{R}^2$ is bounded and continuous. Following \cite{maslowski2000probabilistic}, we denote by $A:=\frac{\partial^2}{\partial x^2}$ the Laplacian operator with Neumamm boundary condition $\frac{du}{dx}(0)=\frac{du}{dx}(1)=0,$ and by $e^{At}$ the semigroup generated by $A$. Then it is known (see (2.30) of \cite{maslowski2000probabilistic}) that the solution to \eqref{wihteboundary} shall be reformulated as 
  
\begin{equation}\label{mildformulation}
    X_t=e^{At}X_0+\int_0^t (A-I)e^{A(t-s)}Nf(X_s)ds+\int_0^t (A-I)e^{A(t-s)}Nd\omega_s,
\end{equation} where $N:\mathbb{R}^2\to H$ is defined as the Neumann map
sending $(\rho_1,\rho_2)\in\mathbb{R}^2$ to the unique solution satisfying
$$\frac{d^2u}{dx^2}-u=0\quad\text{ on }[0,1],\quad \frac{du}{dx}(0)=\rho_1,\quad\frac{du}{dx}(1)=\rho_2.$$

Now it is evident that we can apply Girsanov transform to the formulation \eqref{mildformulation} and modify the drift $f$. Consequently we can solve the following McKean-Vlasov SPDE almost identically as the case of the stochastic heat equation
\begin{equation}
\begin{cases}
    \frac{\partial u}{\partial t} =\frac{\partial^2 u}{\partial x^2},\quad t>0,x\in [0,1],\\
    (\frac{\partial u}{\partial x}(t,0)=\frac{\partial u}{\partial x}(t,1))=G(u(t),\mu_t)+\dot{\omega}(t),\quad t\geq 0,\quad\operatorname{Law}(u(t))=\mu_t,
    \end{cases}
    \end{equation}
under the assumption that $G:H\times \mathcal{P}(H)\to \mathbb{R}^2$ is uniformly bounded, continuous in the first argument, and Lipschitz in the second argument in total variation, which means for some $M>0$, 
$$\|G(x,\mu)-G(x,\nu)\|\leq M \|\mu-\nu\|_{TV},$$ for any $\mu,\nu\in\mathcal{P}(H)$, uniformly over $x\in H$. This proves Theorem \ref{theorem1.71.7}, (2).

\subsubsection{stochastic wave equation}
Denote by $H:=L^2([0,1];\mathbb{R})$,consider the stochastic wave equation 
\begin{equation}\label{canonicalwave}
\begin{cases}
    \frac{\partial^2}{\partial t ^2}y(t)=-\Lambda y(t)dt+dW(t),\\ y(0)=y\in H,\quad \frac{\partial}{\partial t}y(0)=z\in H^{-1}.
\end{cases}
\end{equation}
The operator $\Lambda$ is the Laplacian with Neumann boundary condition 
    $$\Lambda x=-\frac{\partial^2 x}{\partial\xi^2},\quad D(\Lambda)=H^2(0,1)\cap H_0^1(0,1).$$

It is customary (following Example 5.8 of\cite{da2014stochastic}) to regard the stochastic wave equation as a second order system by setting $X(t)=\begin{pmatrix} Y(t)\\ \frac{\partial}{\partial t}Y(t)\end{pmatrix}$ and $X_0=\begin{pmatrix} y\\z\end{pmatrix}.$ The stochastic wave equation can be reformulated, on the Hilbert space $\widetilde{H}:=H\oplus D(\Lambda^{-\frac{1}{2}})$\footnote{With our choice of $H$ and $\Lambda$, we have in our case of interest $\widetilde{H}=L^2((0,1);\mathbb{R})\otimes H^{-1}((0,1),\mathbb{R})$. Here $H^{-1}$ denotes the Sobolev space of order -1.} as 

\begin{equation}
    \label{linearwave}
dX(t)=AX(t)dt+B dW(t),\quad X(0)=X_0\end{equation}
with $$A\begin{pmatrix} y\\z\end{pmatrix}=\begin{pmatrix} 0&1\\-\Lambda&0\end{pmatrix}\begin{pmatrix} y\\z\end{pmatrix}, \quad B(u)=\begin{pmatrix}0\\ u\end{pmatrix}.$$

With this choice of $H$ and $\Lambda$, the stochastic convolution $W_A(\cdot)$, such that $$X(t)=W_A(t)X(0)+\int_0^t W_A(t-s)BdW(s)$$ for $X(t)$ solving \eqref{linearwave}, satisfies the assumption of Theorem 5.2 of \cite{da2014stochastic}. Consequently, $W_A(\cdot)$
is Gaussian, the solution $X(t)$ to \eqref{linearwave} is function-valued, and (by Hypothesis 7.2(iii) and Theorem 10.18 of \cite{da2014stochastic}) we can apply Girsanov transform to the driving Wiener process of \eqref{linearwave}.

Given a bounded measurable function $G:[0,T]\times \widetilde{H}\to \widetilde{H}$ that takes values only on the second component $D(\Lambda^{-\frac{1}{2}})$, Girsanov transform implies that the SPDE
$$d\tilde{X}(t)=A\tilde{X}(t)dt+G(t,\tilde{X}(t))dt+B dW(t),\quad \tilde{X}_0=X_0$$
has a unique probabilistic weak solution. Now assume that we are given a nonlinearity $G:[0,T]\times \widetilde{H}\times\mathcal{P}(\widetilde{H})\to 0\oplus  D(\Lambda^{-\frac{1}{2}})\subseteq \widetilde{H}$ (again taking value only in the second component of $\widetilde{H}$) satisfying 
$$\|B^{-1}G(t,x,\mu)-G(t,x,\nu)\|_H\leq M\|\mu-\nu\|_{TV},\quad\text{ for each }\mu,\nu\in\mathcal{P}(\widetilde{H}),$$
uniformly over $t\in[0,T]$ and $x\in\widetilde{H}$. Following the procedure for stochastic heat equation in Section \ref{she123}, we can do Girsanov transform on the second component of $\widetilde{H}$ and obtain a unique probabilistic weak solution to the McKean-Vlasov SPDE
\begin{equation}\label{mckeantransf}
    dX(t)=A X(t)dt+G(t,X(t),\mu_t)dt+B dW(t),\quad \mu_t=\operatorname{Law}((X(t)).
\end{equation}
Rewriting \eqref{mckeantransf} in the canonical form of wave equation as in \eqref{canonicalwave}, we finish the proof of Theorem \ref{theorem1.71.7}, (3).

\subsection{Quantitative propagation of chaos}\label{chaos7.2}

In this section we generalize the propagation of chaos results in Lacker\cite{lacker2021hierarchies} to the SPDE setting, following closely the argument in that paper. To improve readability, it might be instructive to compare our arguments with that given in Section 1.2 of \cite{lacker2021hierarchies}. We denote by $H:=L^2([0,1];\mathbb{R})$ and assume $f:[0,T]\times H\times H\to H$ is bounded measurable.

Consider $N$ particles $X^1,\cdots,X^N$ with i.i.d. initial law $\mu_0\in\mathcal{P}(H)$ satisfying 
$$\frac{\partial}{\partial t}X^i(t)=\frac{\partial^2}{\partial\sigma^2} X^i(t)dt+\frac{1}{N-1}\sum_{j=1,j\neq i}^N F(t,X^i(t),X^j(t))dt+dW^i(t),\quad i=1,\cdots,N,$$then for $k=1,\cdots,N$ we define a conditioned drift, with $(x^1,\cdots,x^k)\in H^{\otimes k}$, $$
\begin{aligned}
\widehat{F}_i^k(t,x^1,\cdots,x^k)&=\mathbb{E}\left[\frac{1}{N-1}\sum_{j=1,j\neq i}^N F(t,X^i(t),X^j(t)\mid X^1(t)=x_1,\cdots,X^k(t)=x_k\right]\\&=\frac{1}{N-1}\sum_{j\leq k,j\neq i} F(t,X^i(t),X^j(t))+\frac{N-k}{N-1}\langle F(t,X^i(t),\cdot),P_{t,(x_1,\cdots,x^k)}^{(k+1\mid k)}\rangle.\end{aligned},$$
where $P_{t,(x_1,\cdots,x_k)}^{(k+1\mid k)}$ is the law of $X^{k+1}(t)$ given $(X^1(t),\cdots,X^k(t))=(x_1,\cdots,x_k).$

Now we denote by $P_t^{(N,k)}$ the joint law of $(X^1(t),\cdots,X^k(t))$, (by exchangeability of the particle system any choice of the $k$-particles will do), and denote by $\mu_t$ the law of the McKean-Vlasov SPDE 
$$\frac{\partial}{\partial t}Y(t)=\frac{\partial^2}{\partial\sigma^2} Y(t)dt+\langle \mu_t, F(t,Y(t),\cdot)\rangle dt+dW(t),\quad\operatorname{Law}(Y_t)=\mu_t,$$
with the same initial law $\mu_0$.

Recall that the Girsanov density between two SPDEs with the same diffusion coefficient is given in \ref{7.5555}, which has the same expression as in the finite-dimensional SDE setting. Therefore, after exactly the same computation as in \cite{lacker2021hierarchies}, denoting by $H_t^k:=H\left(P_t^{(N,k)}\mid \mu_t^{\otimes k}\right),$ we have 
$$\frac{d}{dt}H_t^k\leq\frac{1}{2}\sum_{i=1}^k\int _{H^{\otimes k}}\left| \widehat{F}_i^k(t,x)-\langle F(t,x_i,\cdot),\mu_t\rangle\right|^2 P_t^{N,k}(dx).$$

Now we expand the squares and the explicit expression of $\widehat{F}_i^k$, using the fact that $F$ is bounded in $H$ and the following easy consequence of boundedness: there exists $C>0$ such that 
$$ \|\langle F(t,x,\cdot),\mu-\nu\rangle\|_H\leq  C \|\mu-\nu\|_{TV},\quad \text{ for all }x\in H,\mu,\nu\in\mathcal{P}(H),$$
then by Pinsker's inequality and the following relation
$$\int_{H^{\otimes k}} H\left(P_{t,x}^{(k+1\mid k)}\mid \mu_t\right)P_t^{(N,k)}dx=H_t^{k+1}-H_t^k,$$
we deduce that we can find positive constants $M$ and $\gamma$ such that
$$\frac{d}{dt}H_t^k\leq \frac{k(k-1)^2}{(N-1)^2} M H_t^k+\gamma k(H_t^{k+1}-H_t^k).$$
It is solved in \cite{lacker2021hierarchies} that, when $k$ is much smaller than $N$, $H_t^k=O(\frac{k^2}{N^2})$. We now arrive at the desired propagation of chaos result. The same result holds if $(X^1(0),\cdots, X^N(0))$ are not i.i.d. but exchangeable and moreover satisfies $H_0^k\leq C\frac{k^2}{N^2}$ for some $C>0$ and all $k=1,\cdots,N$.

\section{Conclusion}

We briefly summarize the strength and weakness of the relative entropy approach highlighted in this paper.

The most dominant advantage is that we can treat infinite dimensional objects just as smoothly as we do with finite dimensional ones. An important example is path dependent drifts as considered in \cite{lacker2021hierarchies}, where the state space $\mathcal{C}([0,T],\mathbb{R}^d)$ is clearly infinitely dimensional. Another important case, put forward in this article, concerns SPDEs with nonsmooth interaction. Here the state space is a Hilbert space such as $L^2([0,1],\mathbb{R}^d)$, which is infinite dimensional as well. Our strategy is well adapted to these settings.

Another feature is that we are possibly using the least amount of information of the system. We do not need to know if the original equation with drift (that is, the SDE $dX_t=b(t,X)dt+dW_t$) has a probabilistic strong solution, if the process is Markovian or not, if it defines a strong Feller or merely Feller Markov process, etc. There are examples where each of these assumptions can fail (for example Remark 5 in Section 3 of \cite{scheutzow1984qualitative}), but we can still apply Girsanov transform and solve the fixed point equation. Other strategies like the Zvonkin transform \cite{zvonkin1974transformation} usually use detailed information of the system, and in many cases the strong well-posedness or the strong Feller property should have been inherited by the system before these techniques can be used.

Moreover, for interactions that have a growth at infinity (so that they do not belong to classical Besov-Hölder spaces unless we impose a weight), relative entropy turns up to be a very good candidate for the weight function, though the proof is in general much longer than the bounded case. This point is elaborated in detail in Remark \ref{linearmore}. Relative entropy also makes it easier to consider interactions of \textit{mixed} smoothness by allowing us to do all the computations on the same platform, see Theorem \ref{lplqfirst} and \ref{wes}. 

One limitation of the entropy approach is that it is restricted to Brownian motion, and does not carry over to $\alpha$-stable processes, $\alpha\in(0,2)$. Moreover, currently there are three important cases where detailed pathwise information of the coefficient is necessary, and relative entropy considerations seem to be insufficient. They are: (1) when the diffusion coefficient depends on the density (see \cite{de2020strong}); (2) when a renormalization procedure is needed to make sense of the solution (see \cite{shen2022large}), and (3) in the particular case that the interaction is given by convolution with a kernel, such that we can solve McKean-Vlasov SDEs with coefficients of rather low regularity (see Remark \ref{a special case} and the very recent work \cite{de2022multidimensional}).

\appendix 
\section{Path space and relative entropy}
\label{appendixas}
\begin{Definition}\label{???}
Fix $k \in \mathbb{N}$. A function $\bar{b}:[0, T] \times \mathcal{C}_{T}^{k} \rightarrow \mathbb{R}$ is said to be progressively measurable if (a) it is Borel measurable and (b) it is non-anticipative, i.e.,  $\bar{b}(t, x)=\bar{b}\left(t, x^{\prime}\right)$ for every $t \in[0, T]$ and $x, x^{\prime} \in C_{T}^{k}$ satisfying $\left.x\right|_{[0, t]}=\left.x^{\prime}\right|_{[0, t]} .$\end{Definition}
For any $k\in\mathbb{N}$, $Q\in\mathcal{P}(\mathcal{C}_T^k)$, and $t\in[0,T]$, let $Q[t]$ be the projection of $Q$ to $\mathcal{C}_t^k$. Also write $Q_t$ as the marginal law of $Q$ at time $t$.

 The relative entropy of $\nu$ with respect to $\nu'$, for two probability measures $(\nu,\nu')$ on a common measurable space, is defined by $$
H\left(\nu \mid \nu^{\prime}\right)=\int \frac{d \nu}{d \nu^{\prime}} \log \frac{d \nu}{d \nu^{\prime}} d \nu^{\prime} \quad \text { if } \nu \ll \nu^{\prime}, \quad \text { and } \quad H\left(\nu \mid \nu^{\prime}\right)=\infty \text { otherwise. }
$$

The data processing inequality of relative entropy (see \cite{lacker2021hierarchies}, (4.3)) implies:
\begin{equation}\label{processing}
    H\left(\mu[t]\mid\nu[t]\right)\leqslant  H\left(\mu[s]\mid\nu[s]\right)\quad \text{ for all } 
    0\leq t\leq s\leq T.
\end{equation}

We will extensively use Pinsker's inequality, which states that for two probability measures $\left(\nu,\nu'\right)$ on a common measurable space, denote by $\|\cdot\|_{TV}$ the total variation distance, then
\begin{equation}\label{variational}\|\nu-\nu'\|_{TV}\leq\sqrt{2H\left(\nu\mid \nu'\right)}.\end{equation}

Suppose for each $i=1,2$ that the $S D E$
$$
d Z_{t}^{i}=b^{i}\left(t, Z^{i}\right) d t+d W_{t}^{i}, \quad t \in[0, T]
$$
admits a weak solution, and let $P^{i} \in \mathcal{P}\left(\mathcal{C}_{T}^{k}\right)$ denote its law. Assuming that 
$$
M_{t}:=\exp \left(\int_{0}^{t}\left(b_{s}^{2}-b_{s}^{1}\right) \cdot d W_{s}^{1}-\frac{1}{2} \int_{0}^{t}\left|b_{s}^{1}-b_{s}^{2}\right|^{2} d s\right), \quad t \in[0, T]
$$
is a martingale under $P^{1}$,
then for $0\leq t\leq T$,
\begin{equation}\label{Browniancase}
H\left(P^{1}[t] \mid P^{2}[t]\right)=H\left(P_{0}^{1} \mid P_{0}^{2}\right)+\frac{1}{2} \mathbb{E} \int_{0}^{t}\left|b^{1}\left(s, Z^{1}\right)-b^{2}\left(s, Z^{1}\right)\right|^{2} d s.
\end{equation}
See Lemma 4.4 of \cite{lacker2021hierarchies} for a general criteria in terms of finite entropy condition.

To check $M$ is a martingale, we may use the following corollary of Novikov's condition: for an adapted process $\{\beta_t\}$ with respect to $(\Omega,\mathcal{F},\mathbb{P})$, assume that for some $\delta>0$, $\sup _{0\leq t \leq T} \mathbb{E}\left[\exp \left(\delta \beta_{t}^{2}\right)\right]<\infty$, then 
 $\varphi_t(\beta):=
\exp \left(\int_{0}^{t} \beta_{s} d W_{s}-\frac{1}{2} \int_{0}^{t} \beta_{s}^{2} d s\right)
$ is a martingale on $[0,T]$.

\section*{Acknowledgements}
I am grateful to my supervisor, James Norris, for reviewing my drafts and offering insightful suggestions. I would also like to thank Ioannis Kontoyiannis for a discussion on relative entropy and Avi Mayorcas for conversations on fractional Brownian motions. 
\printbibliography
\end{document}